\documentclass[10pt]{amsart}
\usepackage{amssymb}
\usepackage{bm}
\usepackage{graphicx}
\usepackage[centertags]{amsmath}
\usepackage{amsfonts}
\usepackage{amsthm}
\usepackage{graphicx}
\newtheorem{thm}{Theorem}
\newtheorem{cor}[thm]{Corollary}
\newtheorem{lem}[thm]{Lemma}
\newtheorem{prop}[thm]{Proposition}
\newtheorem{defn}[thm]{Definition}
\theoremstyle{definition}
\newtheorem{rem}{Remark}
\newtheorem{notation}{Notation}

\newtheorem{examp}{Example}

\newcommand{\rr}{\mathbb{R}}

\newcommand{\ee}{\varepsilon}
\newcommand{\nn}{\mathbb{N}}
\newcommand{\ttt}{\mathcal{T}}

\newcommand{\g}{\mathcal{G}}

\begin{document}

\title[Finite order spreading models]{Finite order spreading models}
\author{S. A. Argyros, V. Kanellopoulos and K. Tyros}
\address{National Technical University of Athens, Faculty of Applied Sciences,
Department of Mathematics, Zografou Campus, 157 80, Athens,
Greece} \email{sargyros@math.ntua.gr} \email{bkanel@math.ntua.gr}
\email{ktyros@central.ntua.gr}

\begin{abstract}
  Extending the classical notion of the spreading
  model, the $k$-spreading models of a Banach space are
  introduced, for every $k\in\nn$. The definition, which is based on the $k$-sequences and plegma
  families, reveals a new class of spreading sequences associated
  to a Banach space. Most of the results of the
  classical theory are stated and proved in the higher order setting. Moreover, new phenomena like the universality of the
  class of the 2-spreading models of $c_0$ and the composition
  property are established. As consequence, a problem concerning
  the structure of the $k$-iterated spreading models is solved.
\end{abstract}

\thanks{This research is partially supported
by NTUA Programme PEBE 2009 and it is part of the PhD Thesis of the
third named author}

\maketitle

\section*{Introduction}
The present work was motivated by a problem of E. Odell and Th.
Schlumprecht concerning the structure of the $k$-iterated
spreading models of the Banach spaces. Our attempt to answer the
problem led to the $k$-spreading models which in turn are based on
the $k$-sequences and plegma families. The aim of this paper is to
introduce the above concepts and to develop a theory yielding ,
among others, a solution to the aforementioned problem.

Spreading models, invented by A. Brunel and L. Sucheston (c.f.
\cite{BS}), posses a key role in the modern Banach space theory.
Let us recall that a spreading model of a Banach space $X$ is a
spreading sequence\footnote{A sequence $(e_n)_{n}$ in a seminormed
space $(E,\|\cdot\|_*)$ is called spreading if for every
$n\in\nn$, $k_1<\ldots<k_n$ in $\nn$ and $a_1,\ldots,a_n\in\rr$ we
have that $\|\sum_{j=1}^na_j e_j\|_*=\|\sum_{j=1}^n a_j
e_{k_j}\|_*$.

In the literature the term ``spreading model" usually indicates
the space generated by the corresponding spreading sequence rather
than the sequence itself. We have chosen to use the term for the
spreading sequence and whenever we refer to $\ell^p$ or $c_0$
spreading model we shall mean that the spreading sequence is
equivalent to the usual basis of the corresponding space.}
generated by a sequence of $X$. The spreading sequences have
regular structure and the spreading models act as the tool for
realizing that structure in the space $X$ in an asymptotic manner.
This together with the Brunel-Sucheston's discovery that every
bounded sequence has a subsequence generating a spreading model
determine the significance and importance of this concept. For a
comprehensive presentation of the theory of the spreading models
we refer the interested reader to the monograph of  B. Beauzamy
and J.-T. Laprest\'e (c.f. \cite{BL}).

Iteration is naturally applicable to spreading models. Thus one
could define the 2-iterated spreading models of a Banach space $X$
to be the spreading sequences which occur as spreading models of
the spaces generated by spreading models of $X$. Further iteration
yields the $k$-iterated spreading models of $X$, for every
$k\in\nn$. Iterated spreading models appeared in the literature
shortly after Brunel-Sucheston's invention. Indeed, B. Beauzamy
and B. Maurey in \cite{BM}, answering a problem of H.P. Rosenthal,
showed that the class of the 2-iterated spreading models does not
coincide with the corresponding one of the spreading models. In
particular they constructed a Banach space admitting the usual
basis of $\ell^1$ as a $2$-iterated spreading model and not as a
spreading model.

E. Odell and Th. Schlumprecht in  \cite{O-S} asked whether or not
every Banach space admits a $k$-iterated spreading model equivalent
to the usual basis of $\ell^p$, for some $1\leq p<\infty$, or $c_0$.
Let us also point out that in the same paper they provided a
reflexive space $\mathfrak{X}$ with an unconditional basis such that
no $\ell^p$ or $c_0$ is embedded into the space generated by any
spreading model of the space. This remarkable result answered a long
standing problem of the Banach space theory.

Our approach uses the $k$-spreading models which in many cases
include the $k$-iterated ones. The $k$-spreading models are always
spreading sequences $(e_n)_n$ in a seminormed space $E$. They are
generated by $k$-sequences $(x_s)_{s\in[\nn]^k}$, where $[\nn]^k$
denotes the family of all  $k$-subsets of $\nn$. A critical
ingredient in the definition is the plegma families $(s_i)_{i=1}^l$
of elements of $[\nn]^k$, described as follows.

A finite sequence $(s_j)_{j=1}^l$ in $[\nn]^k$ is a plegma family if
its elements satisfy the following order relation: for every $1\leq
i\leq k$,  $s_1(i)<\ldots<s_l(i)$ and for every $1\leq i< k$,
$s_l(i) < s_1(i+1)$. The plegma families, as they are used in the
definition, force a weaker asymptotic relation of the $k$-spreading
models to the space $X$, as $k$ increases. For $k=1$, the plegma
families coincide to the finite subsets of $\nn$ yielding that the
new definition of the 1-spreading models recovers the classical one.
For $k>1$, the plegma families have a quite strict behavior which is
described in the first section of the paper. Of independent interest
is also Lemma \ref{intro_denslem} stated below.

The $k$-spreading models of a Banach space $X$ are denoted by
$\mathcal{SM}_k(X)$ and they define an increasing sequence. As the
definition easily yields, the same holds for the $k$-iterated ones.
Similarly to the classical case, for every bounded $k$-sequence
$(x_s)_{s\in[\nn]^k}$ there exists an infinite subset $L$ of $\nn$
such that the $k$-subsequence $(x_s)_{s\in[L]^k}$ generates a
$k$-spreading model.

The advantage of the $k$-spreading models is that, unlike the
$k$-iterated ones, for $k\geq2$, the space $X$ determines directly
their norm, through the $k$-sequences. Moreover, the $k$-spreading
models have a transfinite extension yielding a hierarchy of
$\xi$-spreading models for all $\xi<\omega_1$. The definition and
the study of this hierarchy is more involved and will be presented
elsewhere. We should also mention that L. Halbeisen and E. Odell
(c.f. \cite{H-O}) introduced the asymptotic models which share
some common features with the 2-spreading models. The asymptotic
models are associated to bounded 2-sequences $(x_s)_{s\in[\nn]^2}$
and they are not necessarily spreading sequences.

The paper mainly concerns the definition and the study of the
$k$-spreading models.
Highlighting the results of the paper we should mention the
universal property satisfied by the 2-spreading models of $c_0$.
More precisely, it is shown that every spreading sequence is
isomorphically equivalent to some 2-spreading model of $c_0$. As the
spaces generated by $k$-iterated spreading models of $c_0$ are
isomorphic to $c_0$, the previous result shows that the
$k$-spreading models do not coincide with the $k$-iterated ones. The
composition property is also established. Roughly speaking,
 under some natural conditions, the
$d$-spreading model of a $k$-spreading model of a Banach space $X$
is a $(k+d)$-spreading model of $X$. This result is used for showing
that a special class of the $k$-iterated spreading models are
actually $k$-spreading models. We also extend to the higher order
results of the spreading model theory. Among others we provide
conditions for the $k$-sequences to generate unconditional spreading
models and we study properties like non-distortion and duality of
$\ell^1$ and $c_0$ $k$-spreading models. Moreover we introduce the
Ces\`aro summability for $k$-sequences and we prove the following
that extends a classical theorem due to H.P. Rosenthal (c.f.
\cite{M,Ro}).
\begin{thm}
Let $X$ be a Banach space, $k\in\nn$ and  $(x_s)_{s\in[\nn]^k}$ be
a weakly relatively compact $k$-sequence in $X$, i.e.
$\overline{\{x_s:s\in[\nn]^k\}}^w$ is $w$-compact. Then there
exists $M\in [\nn]^\infty$ such that at least one of the following
holds:
\begin{enumerate}
\item[(1)] The subsequence  $(x_s)_{s\in[M]^k}$ generates a
$k$-spreading model equivalent to the usual basis of $\ell^1$.
\item[(2)] There exists $x_0\in X$ such that for every $L\in
[M]^\infty$, $(x_s)_{s\in [L]^k}$ is $k$-Ces\`aro summable to
$x_0$.
\end{enumerate}
\end{thm}
There are significant differences between the cases $k=1$ and
$k\geq2$. First for $k=1$ the two alternatives are exclusive which
does not remain valid for $k\geq 2$. Second the proof for the case
$k\geq 2$ uses the following density result concerning plegma
families which is a consequence of the multidimensional
Szemeredi's theorem due to H. Furstenberg and Y. Katznelson (c.f.
\cite{FK}).
\begin{lem}\label{intro_denslem}
Let  $\delta>0$ and $k, l\in\nn$. Then  there exists  $n_0\in \nn$
such that for every $n\geq n_0$ and   every subset $\mathcal{A}$ of
the set of all $k$-subsets of $\{1,\ldots, n\}$ of size at least
$\delta (\substack{n\\ k})$, there exists  a plegma $l$-tuple
$(s_j)_{j=1}^l$ in $\mathcal{A}$.
  \end{lem}

We close the paper with two examples. The first one is a Banach
space similar to the aforementioned one of Odell-Schlumprecht. It is
proved that no $k$-spreading model of the space is isomorphic to
some $\ell^p$, $1\leq p<\infty$, or $c_0$. The composition property,
mentioned above, yields that the same holds for the $k$-iterated
spreading models and thus the answer to the aforementioned
Odell-Schlumprecht problem is a negative one. In the second example,
for every $k\in\nn$ we present a space $\mathfrak{X}_{k+1}$
admitting the usual basis of $\ell^1$ as a $(k+1)$-spreading model
while for every $d\leq k$, $\mathfrak{X}_{k+1}$ does not admit
$\ell^1$ as a $d$-spreading model. As we have mentioned, the
corresponding problem for $k$-iterated spreading models has been
answered in \cite{BM} for $k+1=2$. It seems that for $k>1$ this
problem is still open. However, recently the $(k+1)$-iterated spreading models have been separated by the $k$ ones in \cite{AM}. The proofs  in both examples make use of the
results exhibited in the previous sections of the paper.
\subsection*{Notation} By $\nn=\{1,2,...\}$
we denote the set of all positive integers.  We will use capital
letters as $L,M,N,...$ (resp. lower case letters as  $s,t,u,...$)
to denote infinite subsets (resp. finite subsets) of $\nn$. For
every infinite subset $L$ of $\nn$, the notation $[L]^\infty$
(resp. $[L]^{<\infty}$) stands for the set of all infinite (resp.
finite) subsets of $L$.  For every $s\in[\nn]^{<\infty}$, by
$|s|$ we denote  the cardinality of $s$. For $L\in[\nn]^\infty$
and $k\in\nn$,  $[L]^k$ (resp. $[L]^{\leq k}$) is the set of all
$s\in[L]^{<\infty}$ with $|s|=k$ (resp. $|s|\leq k$).  For every
$s,t\in[\nn]^{<\infty}$,  we write $s<t$ if either at least one of
them is the empty set, or $\max s<\min t$.

Throughout the paper we shall identify strictly increasing sequences
in $\nn$ with their corresponding range, i.e. we view every strictly
increasing sequence in $\nn$ as a subset of $\nn$ and conversely
every subset of $\nn$ as the sequence resulting from the increasing
ordering of its elements. Thus, for an infinite subset
$L=\{l_1<l_2<...\}$ of $\nn$  and  $i\in\nn$, we set $L(i)=l_i$ and
similarly, for a finite subset $s=\{n_1<..<n_k\}$ of $\nn$ and for
$1\leq i\leq k$, we set $s(i)=n_i$.  Also, for every  $L, N\in
[\nn]^\infty$ and $s\in [\nn]^{<\infty}$,  we set
$L(N)=\{L(N(i)):i\in\nn\}$ and $L(s)=\{L(s(i)): 1\leq i\leq |s|\}$.
Similarly, for every $s\in [\nn]^{k}$ and $ F\subseteq \{1,...,k\}$,
we set $s(F)=\{s(i):i\in F\}$. Also  for $1\leq m\leq k$, we set
$s|m=\{s(i): 1\leq i\leq m\}$.

For every $s,t\in[\nn]^{<\infty}$, we write $s\sqsubseteq t$ (resp.
$s\sqsubset t$) to denote that $s$ is an initial (resp.
\emph{proper} initial) segment of $t$.
Given two  sequences $(s^1_j)_{j=1}^{l_1}$ and $(s^2_j)_{j=1}^{l_2}$
in $[\nn]^{<\infty}$, by
$(s^1_j)_{j=1}^{l_1\;\;\smallfrown}(s^2_j)_{j=1}^{l_2}$, we denote
their concatenation. Similarly for more than two sequences.

For a Banach space $X$ with a
Schauder basis $(e_n)_n$ and every $x\in X$, $x=\sum_n \lambda_n
e_n$ we write $\text{supp}(x)$ to denote the support of $x$, i.e.
$\text{supp}(x)=\{n\in\nn:\lambda_n\neq 0\}$. If the
support of $x$  is finite and $E\subseteq \nn$ then by $E(x)$, we
denote the restriction of $x$ to $E$, namely $E(x)=\sum_{n\in
E}\lambda_n e_n$.

Two sequences $(x_n)_n$ and $(y_n)_n$, not necessarily in the same Banach space, will be called isometric (resp. equivalent) if (resp. there exists $0<c\leq C$ such that) for every $n\in\nn$ and $a_1,\ldots,a_n\in\rr$ we have that $\|\sum_{i=1}^na_ix_i\|=\|\sum_{i=1}^na_iy_i\|$ (resp. $c\|\sum_{i=1}^na_ix_i\|\leq\|\sum_{i=1}^na_iy_i\|\leq C\|\sum_{i=1}^na_ix_i\|$).  Generally concerning Banach space theory the notation and the terminology that we follow is the standard one (see \cite{AK} and  \cite{Lid-Tza}).

\section{Plegma families in $[\nn]^k$}
As we have already mentioned, the basic ingredients of the
definition of the $k$-spreading models are the $k$-sequences and the
plegma families.
In this section we introduce  the  plegma families as well as the
related notions of the plegma paths and the plegma preserving maps.

\subsection{Definition and basic properties}\label{section
admissibility} We start with the definition of the plegma families.
\begin{defn}\label{defn plegma} Let $k\in\nn$ and $M\in [\nn]^\infty$. A
plegma family in $[M]^k$ is a finite sequence  $(s_j)_{j=1}^l$ in
$[M]^k$ satisfying the following properties.
\begin{enumerate}
\item[(i)] For every $1\leq
i\leq k$,  $s_1(i)<\ldots<s_l(i)$.
\item[(ii)] For every $1\leq i< k$,
$s_l(i) < s_1(i+1)$.
\end{enumerate}
For each $l\in \nn$, the set of all sequences $(s_j)_{j=1}^l$ which
are plegma families in $[M]^k$ will be denoted by
$\textit{Plm}_l([M]^k)$. We also set
$\textit{Plm}([M]^k)=\bigcup_{l=1}^\infty\textit{Plm}_l([M]^k)$.
\end{defn}
Notice that  for $l=1$ and every $k\in \nn$,   we have
$\textit{Plm}_1([M]^k)=[M]^k$. Moreover, for $k=1$ and every
$l\in\nn$, $\textit{Plm}_l([M]^1)=[M]^{l}$. In the sequel the
elements of $\textit{Plm}_2([M]^k)$ will be called  plegma
pairs in $[M]^k$.

\begin{rem} Although the notion of the plegma family is natural, it does
not seem to have appeared in the literature. As it was pointed out
to us by S. Todorcevic, a concept that slightly reminds plegma pairs
in $[\nn]^3$ is given by E. Specker in \cite{Sp}.
\end{rem}

  In the next proposition we gather
some useful properties of plegma families. The proof is
straightforward.

\begin{prop} \label{rem34} Let $k,l\in\nn$, $M\in[\nn]^\infty$ and $(s_j)_{j=1}^l$ be a finite sequence  in $[M]^k$.
\begin{enumerate}
\item[(i)] $(s_j)_{j=1}^l\in \textit{Plm}_l([M]^k)$ if and only
if there exists $F\in[M]^{kl}$ such that $s_j(i)=F((i-1)k+j)$,
for every $1\leq i\leq k$ and $1\leq j\leq l$.
\item[(ii)] If $(s_j)_{j=1}^l\in \textit{Plm}_l([M]^k)$
then $(s_{j_p})_{p=1}^m\in\textit{Plm}_m([M]^k)$, for every
$1\leq m\leq l$ and $1\leq j_1<\ldots<j_m\leq l$.
\item[(iii)]  $(s_j)_{j=1}^l\in\textit{Plm}_l([M]^k)$ if and only if
 $(s_{j_1},s_{j_2})$ is a plegma pair in $[M]^k$, for every
 $1\leq j_1<j_2\leq l$.
\item[(iv)] If
$(s_j)_{j=1}^l\in\textit{Plm}_l([M]^k)$ then
$(s_j(F))_{j=1}^l\in\textit{Plm}_l([M]^{|F|})$, for every
non empty $ F\subseteq \{1,...,k\}$.
\end{enumerate}
\end{prop}

\begin{thm} \label{ramseyforplegma}
Let $M$ be an infinite subset of $\nn$ and  $k,l\in\nn$. Then for
every finite partition $\textit{Plm}_l([M]^k)=\bigcup_{j=1}^p
P_j$, there exist $L\in[M]^\infty$ and $1\leq j_0\leq p$ such that
$\textit{Plm}_l([L]^k)\subseteq P_{j_0}$.
\end{thm}
\begin{proof}
By Proposition \ref{rem34} (i), we conclude that  the map sending
each plegma family $(s_j)_{j=1}^l$ in $[M]^k$  to its union
$\bigcup_{j=1}^ls_j$ is a bijection from
$\textit{Plm}_l([M]^k)$ onto $[M]^{kl}$. Therefore the
partition of $\textit{Plm}_l([M]^k)$ induces a corresponding
one to $[M]^{kl}$ and the conclusion easily follows by  applying
the Ramsey's theorem \cite{R}.
\end{proof}

\subsection{Plegma paths in $[\nn]^k$}
In this subsection we introduce the definition of the plegma paths.
As we shall see in the sequel, the plegma paths play important role
in the development of the theory of $k$-spreading models.
\begin{defn}
Let $l,k\in\nn$ and $M\in[\nn]^\infty$. We will say that a finite
sequence $(s_j)_{j=0}^l$  is a plegma path of length $l$ from $s_0$
to $s_l$ \textit{in} $[M]^k$, if $(s_{j-1},s_{j})$ is a plegma pair
in $[M]^k$, for every $1\leq j\leq l-1$.
\end{defn}

\begin{lem}\label{lemma conserning the length of the plegma path}
  Let $k\in\nn$ and $(s_j)_{j=0}^l$ be a plegma
  path in $[\nn]^k$.  If $s_0<s_l$ then $l\geq k$.
\end{lem}
\begin{proof}
 Suppose on the contrary that $s_0<s_l$ and $l<k$.
 Since $(s_{j-1}, s_j)$ is a plegma pair in
$[\nn]^k$,  we have $s_j(i_1)<s_{j-1}(i_2)$, for every $1\leq j\leq
l$ and $1\leq i_1<i_2\leq k$. Hence,
$s_l(1)<s_{l-1}(2)<s_{l-2}(3)<\ldots<s_0(l+1)\leq s_0(k)$, which
contradicts that $s_0<s_l$.
\end{proof}

\begin{defn}
Let  $k\in\nn$ and $M\in[\nn]^\infty$. An $s\in [M]^k$ will be
called skipped in $M$ if for every  $1\leq i<k$ there exists $m\in
M$ such that $s(i)<m<s(i+1)$. The set of all skipped $s\in [M]^k$ in
$M$ will be denoted by $[M]^k_\shortparallel$.
\end{defn}

\begin{rem}\label{rem453} Notice that for every $m\in\nn$
and  $s\in [M]^k_\shortparallel$  there exists a plegma path
$(s_j)_{j=0}^l$ in $[M]^k$ with $s_0=s$.
\end{rem}

\begin{prop}\label{accessing everything with plegma path of length |s_0|}
Let $k\in\nn$ and $M\in[\nn]^\infty$. Then for every
$s,t\in[M]_\shortparallel^k$ with $s<t$ there exists a plegma path
of length $k$ in $ [M]^k $  from $s$ to $t$. Moreover, every plegma
path in $[\nn]^k$ from $s$ to $t$ has length at least $k$.
\end{prop}
\begin{proof} Fix $s,t\in[M]_\shortparallel^k$ with $s<t$. It is clear that we may
choose $\tilde{s}, \tilde{t}\in [M]^{2k-1}$ such that
$\tilde{s}(2i-1)=s(i)$ and similarly $\tilde{t}(2i-1)=t(i)$, for
every $1\leq i\leq k$. For every $0\leq j\leq k$, we  set
\[s_j=\big\{\tilde{s}(2i-1+j):1\leq i\leq k-j\big\}\cup\big\{\tilde{t}(2i-1+k-j):1\leq i\leq j\big\}\]
It is easy to check that $s_0=s$, $s_k=t$ and $(s_j)_{j=0}^k$ is a
plegma path in $[M]^k$. Moreover, by Lemma \ref{lemma conserning the
length of the plegma path}, every plegma path in $[\nn]^k$ from $s$
to $t$ is of length at least $k$. Hence $(s_j)_{j=0}^k$ is a plegma
path from $s$ to $t$ in $[M]^k$ with the least possible length and
the proof is complete.
\end{proof}
\begin{rem}\label{graphs} In terms of graph theory the
above proposition states that in the directed graph with vertices
the elements of  $ [\nn]^k$ and edges the plegma pairs $(s,t)$ in
$[\nn]^k$, the distance between two vertices $s$ and $t$ with $s<t$
is equal to $k$.
\end{rem}
\subsection{Plegma families and mappings}
\begin{defn}
  Let $k_1,k_2\in\nn$, $M\in[\nn]^\infty$ and $\varphi:[M]^{k_1}\to
  [\nn]^{k_2}$. We will say that the map $\varphi$ is plegma
  preserving from  $[M]^{k_1}$ into $[\nn]^{k_2}$ if for every plegma family $(s_j)_{j=1}^l$ in
  $[M]^{k_1}$,
    $(\varphi(s_j))_{j=1}^l$ is a plegma family in $[\nn]^{k_2}$.
\end{defn}

\begin{rem}\label{rem3456} Let $k_1,k_2\in\nn$. If $k_1<k_2$ then for every $M\in[\nn]^\infty$ there exists a plegma preserving map from
$[M]^{k_2}$ onto $[M]^{k_1}$. For instance, by Proposition
\ref{rem34},  the map  $s\to s|k_1$ is plegma preserving from
$[M]^{k_2}$ onto $[M]^{k_1}$.
\end{rem}
In contrast to the above remark we have the following.
\begin{thm}\label{non plegma preserving maps}
Let $k_1,k_2\in\nn$.   If $k_1<k_2$ then for every $M\in[\nn]^\infty$
and  $\varphi:[M]^{k_1}\to [\nn]^{k_2}$  there exists $L\in[M]^\infty$
such that for every plegma pair $(s_1,s_2)$ in $[ L]^{k_1}$ neither
$(\phi(s_1),\phi(s_2))$ nor $(\phi(s_2),\phi(s_1))$ is a plegma pair in $[\nn]^{k_2}$.
 In particular, there exists no  $L\in[M]^\infty$ such that the map $\varphi$ is plegma preserving from $[L]^{k_1}$ into $[\nn]^{k_2}$.
\end{thm}

\begin{proof} Let $M\in[\nn]^\infty$ and $\varphi:[M]^{k_1}\to [\nn]^{k_2}$.
 We set $P_1$ (resp. $P_2$) to be the set of all $(s_1,s_2)\in
\textit{Plm}_2([M]^{k_1})$ such that
$(\varphi(s_1),\varphi(s_2))$ (resp.
$(\varphi(s_2),\varphi(s_1))$) is a plegma pair in $[\nn]^{k_2}$
and $P_3=\textit{Plm}_2([M]^{k_1})\setminus (P_1\cup P_2)$.
By Theorem \ref{ramseyforplegma} there exist $i\in\{1,2,3\}$ and
$L\in [M]^\infty$ such that
$\textit{Plm}_2([L]^{k_1})\subseteq P_i$. It remains to show
that $i=3$.

Indeed, assume that $i= 2$. By Remark \ref{rem453} we may choose a
plegma path $(s_j)_{j=0}^l$ in $[L]^k$ with $min(\varphi(s_0))<l$.
For every $0\leq j\leq l$, we set  $n_j=\min(\varphi(s_j))$. Since
$\textit{Plm}_2([L]^{k_1})\subseteq P_2$, we have that
$(n_j)_{j=0}^l$ is a strictly decreasing sequence in $\nn$ with
length $l+1$. Since $n_0<l$ this is impossible.

It remains to show that $i\neq 1$. Indeed, assume on the contrary.
Then notice that $\varphi$ transforms every plegma path in
$[L]^{k_1}$ to a plegma path of equal length  in $[\nn]^{k_2}$.
Using Remark \ref{rem453}, it is easy to
see that we may choose $s<t$ in $[L]_\shortparallel^{k_1}$ such
that $\varphi(s)<\varphi(t)$ and $\varphi(s), \varphi(t)\in
[\nn]_\shortparallel^{k_2}$. By Proposition \ref{accessing
everything with plegma path of length |s_0|} and Remark
\ref{graphs}, we have that the distance of $s,t$ is equal to $k$
while that of $\varphi(s), \varphi(t)$ is equal to $k_2$. But
since $s,t$ are joined by a plegma path of length $k_1$ and
$\varphi$ preserves plegma paths we have that the distance of
$\varphi(s), \varphi(t)$ is at most $k_1$. Hence $k_2\leq k_1$,  a
contradiction.
\end{proof}

\begin{prop}\label{lemma making a hereditary nonconstant function, nonconstant on plegma pairs}
Let $A$ be a set, $k\in\nn$, $M\in[\nn]^\infty$ and
$\varphi:[M]^k\to A$. Then  there exists $L\in [M]^\infty$ such
that either the restriction of $\varphi$ on $[L]^k$ is constant or
for every plegma pair $(s_1,s_2)$ in $[ L]^k$,
$\varphi(s_1)\neq\varphi(s_2)$.
\end{prop}
\begin{proof}
By Theorem \ref{ramseyforplegma} there exists $N\in [M]^\infty$ such
that exactly one of the following are satisfied.
\begin{enumerate}
\item[(i)] For every  plegma pair
$(s_1,s_2)$ in $[N]^k$, $\varphi(s_1)=\varphi(s_2)$.
\item[(ii)] For every  plegma pair $(s_1,s_2)$ in $[ N]^k$,
$\varphi(s_1)\neq\varphi(s_2)$.
\end{enumerate}
Therefore, it suffices to show that  the first alternative  implies
that there exists $L\in [N]^\infty$ such that $\varphi$ is constant
on $[L]^k$. Indeed,  let $s=(N(2),N(4),...,N(2k))$, $L=\{N(2n):
n\geq k+1\}$ and $t\in [L]^k$. Observe that  $s<t$ and
$s,t\in[N]^k_\shortparallel$ and therefore, by Proposition
\ref{accessing everything with plegma path of length |s_0|},  there
exists a plegma path $(s_j)_{j=0}^k$ of length $k$  in $[N]^k$ with
$s_0=s$ and $s_k=t$. Assuming that (i) holds,  we get that
\[\varphi(s)=\varphi(s_0)=\varphi(s_1)=\ldots=\varphi(s_k)=\varphi(t)\]
Hence for every $t\in[L]^k$, $\varphi(t)=\varphi(s)$, i.e. $\varphi$
is constant on $[L]^k$.
  \end{proof}

\section{Spreading sequences}
We recall that a sequence $(e_n)_{n}$ in a seminormed linear space
$(E,\|\cdot\|_*)$ is called \emph{spreading} if it is isometric to
any of its subsequences, i.e. for every $n\in\nn$,
$a_1,\ldots,a_n\in \rr$ and $k_1<\ldots<k_n$ in $\nn$ we have that
$\|\sum_{j=1}^na_je_j\|_*=\|\sum_{j=1}^na_je_{k_j}\|_*$. In this
section we will briefly discuss the norm properties of the
spreading sequences. The interested reader can find a detailed
analysis  in the monographs \cite{AK} and \cite{BL}.

 The proof of the following
result shares similar ideas with the one of Proposition I.1.B.2 in
\cite{BL}.

\begin{prop}\label{sing}
Let $(E,\|\cdot\|_*)$ be a seminormed linear space and $(e_n)_{n}$
be a spreading sequence in $E$. Then the following are equivalent.
\begin{enumerate}
\item[(i)] There exist $n\in\nn$ and $a_1,\ldots,a_n\in\rr$ not
all zero, with $\|\sum_{i=1}^na_ie_i\|_*=0$.
\item[(ii)] For every $n,m\in\nn$, $\|e_n-e_m\|_*=0$. \item[(iii)]
For every $n\in\nn$ and $a_1,\ldots,a_n\in\rr$,  $
\|\sum_{i=1}^na_ie_i\|_*=|\sum_{i=1}^na_i|\cdot \|e_1\|_*$.
\end{enumerate}
\end{prop}
Spreading sequences in seminormed linear spaces satisfying
(i)-(iii) of the above proposition will be called \emph{trivial}.
By (i) we have that if $(e_n)_{n}$ is non trivial, then
$(e_n)_{n}$ is linearly independent and the restriction of the
seminorm $\|\cdot\|_*$ to  the linear subspace of $E$ generated by
$(e_n)_n$ is actually norm. Therefore,  every non trivial
spreading sequence generates a Banach space.

We classify the non trivial spreading sequences into the following
three categories:
\begin{enumerate}
\item The \emph{singular} spreading sequences, i.e. the non
trivial spreading sequences  which are  not Schauder basic
sequences. \item the \emph{unconditional} spreading sequences and
\item the \emph{conditional Schauder basic} spreading sequences,
i.e.  the non trivial spreading sequences which are Schauder basic
but not unconditional.
\end{enumerate}

The next two results are   restatements of Propositions I.1.4 and
I.4.2 of \cite{BL} respectively.

\begin{prop}\label{equiv forms for 1-subsymmetric weakly null}
Let $(e_n)_{n}$ be a non trivial spreading sequence. Then the
following are equivalent.
\begin{enumerate}
\item[(i)] $(e_n)_{n}$ is unconditional and not
equivalent to the usual basis of $\ell^1$.
\item[(ii)]
$(e_n)_{n}$ is weakly null.
\item[(iii)] $(e_n)_{n}$ is Ces\'aro
summable to zero.
\item[(iv)] $(e_n)_{n}$ is 1-unconditional and not
equivalent to the usual basis of $\ell^1$.
\end{enumerate}
\end{prop}

\begin{prop}\label{thmsingular}
Let $(e_n)_n$ be a non trivial spreading sequence and $E$ the Banach
space generated by $(e_n)_n$. Then  $(e_n)_n$ is singular if and
only if $(e_n)_n$ is weakly convergent to a nonzero element $e\in
E$.
\end{prop}

\begin{rem}\label{properties of the natural decomposition}
Let $(e_n)_n$ be a singular spreading sequence. By the above
proposition, we have that $(e_n)_n$ is of the form $e_n=e'_n+e$,
where $e$ is nonzero and $(e'_n)_n$ is weakly null. This
decomposition of $(e_n)_n$ as $e_n=e'_n+e$ will be called  \emph{the
natural decomposition} of $(e_n)_n$. It is easy to check that
$(e'_n)_{n}$ is non trivial, spreading and not equivalent to the
usual basis of $\ell^1$. Hence by Proposition \ref{thmsingular},
$(e'_n)_n$ is unconditional, weakly null and Ces\`aro summable to
zero. Moreover, if $E$ and $E'$ are the Banach spaces generated by
the sequences $(e_n)_n$ and $(e'_n)_n$ respectively, then $E,E'$ are
isomorphic and $E=E'\oplus<e>$.
\end{rem}

Finally for the conditional Schauder basic spreading sequences we
have the next characterization, which is a consequence of the above
results and Rosenthal's $\ell^1$ theorem \cite{Ro}.

\begin{prop} Let $(e_n)_n$ be a spreading non trivial
sequence and $E$ be the Banach space generated by $(e_n)_n$. Then
$(e_n)_n$ is a conditional Schauder basic sequence if and only if
$(e_n)_n$ is non trivial weak-Cauchy.
\end{prop}

\section{$k$-sequences and $k$-spreading models}
In this section we present the definition of the $k$-sequences and
we introduce the notion of the $k$-spreading models, for all
$k\in\nn$. As we will see, for $k=1$, the definition coincides with
the classical one of A. Brunel and L. Sucheston \cite{BS}.
\subsection{Definitions and basic properties}
We start with the definition of the $k$-sequences.
\begin{defn} Let  $k\in\nn$ and  $X$ be a non empty set. A $k$-sequence in $X$
is  a map $\varphi:[\nn]^k\to X$. A $k$-subsequence in $X$ is a map of the form
 $\varphi:[M]^k\to X$ , where $M\in [\nn]^\infty$.
\end{defn}
 A $k$-sequence $\varphi:[\nn]^k\to X$
will be usually denoted by $(x_s)_{s\in [\nn]^k}$, where
$x_s=\varphi(s)$,  $s\in [\nn]^k$. Similarly, the notation
$(x_s)_{s\in [M]^k}$ stands for the $k$-subsequences
$\varphi:[M]^k\to X$.

\begin{defn}\label{Definition of spreading model}
Let   $X$ be a Banach space, $k\in\nn$, $(x_s)_{s\in [\nn]^k}$ be a
$k$-sequence in $X$ and $(E,\|\cdot\|_*)$ be an infinite dimensional
seminormed linear space with Hamel basis $(e_n)_{n}$. Also let
$M\in[\nn]^\infty$ and $(\delta_n)_n$ be a null sequence of positive
reals. We will say that the $k$-subsequence $(x_s)_{s\in [M]^k}$
generates $(e_n)_{n}$ as a spreading model  as a $k$-spreading model
(with respect to $(\delta_n)_n$), if the following is satisfied.

 For every $m,l\in \nn$, with $m\leq l$, every
$(s_j)_{j=1}^m\in\textit{Plm}_m([M]^k)$ with $s_1(1)\geq M(l)$
and every choice of $a_1,...,a_m\in[-1,1]$, we have
\begin{equation}\label{rsm}\Bigg{|}\Big{\|}\sum_{j=1}^m a_j
x_{s_j}\Big{\|}-\Big{\|}\sum_{j=1}^m a_j
  e_j\Big{\|}_* \Bigg{|}\leq\delta_l\end{equation}
\end{defn}

Since $\text{Plm}([\nn]^1)=[\nn]^{<\infty}$, it is clear that for
$k=1$, Definition \ref{Definition of spreading model} coincides with
the classical definition of a spreading model of an ordinary
sequence $(x_n)_n$ in a Banach space $X$. Thus the $1$-spreading
models are the usual ones.  Moreover, it is easy to see  that for
every $k\in\nn$, every $k$-spreading model $(e_n)_{n}$ is a
spreading sequence.

Let's point out here that  there  exist $k$-sequences in Banach
spaces which generate $k$-spreading models which are trivial
spreading sequences, in other words (see Proposition \ref{sing}),
$\|\cdot\|_*$ is not a norm. For instance, this occurs for every
constant $k$-sequence $(x_s)_{s\in[\nn]^k}$. We should also point
out that even if $(e_n)_n$ is non trivial, it is not necessarily a
Schauder basic sequence. More information on this issue are
contained in Section \ref{s5}.

In the next proposition we state some stability properties of the
$k$-spreading models. The proof is straightforward.
\begin{prop}\label{remark on the definition of spreading model}
Let $k\in\nn$, $(x_s)_{s\in [\nn]^k}$ be a $k$-sequence in a
Banach space $X$,  $M\in[\nn]^\infty$ and $(\delta_n)_n$ be a null
sequence of positive reals.  If $(x_s)_{s\in [M]^k}$ generates a
sequence $(e_n)_{n}$ as a $k$-spreading model with respect to
$(\delta_n)_n$ then the following are satisfied.
\begin{enumerate}
\item[(i)]  For every $L\in[M]^\infty$,
 $(x_s)_{s\in [L]^k}$ generates $(e_n)_{n}$ as a $k$-spreading
model with respect to $(\delta_n)_n$. \item[(ii)] For every  null
sequence $(\delta'_n)_{n}$ of positive reals there exists
$M'\in[M]^\infty$ such that $(x_s)_{s\in [M']^k}$ generates
$(e_n)_{n}$ as a $k$-spreading model with respect to
$(\delta'_n)_{n}$. \item[(iii)] The $k$-sequence
$(y_s)_{s\in[\nn]^k}$, defined by $y_s=x_{M(s)}$,  $s\in[\nn]^k$,
generates $(e_n)_{n}$ as a $k$-spreading model with respect to
$(\delta_n)_n$.
\end{enumerate}
\end{prop}
Let us also notice that for $k=1$ the assertion that (\ref{rsm})
holds for all $m\leq l$ is redundant. This is not the case for
$k\geq 2$, since a plegma family in $[\nn]^k$ is not always a
subsequence of a larger one. However, the next lemma shows that we
may bypass this extra condition by passing to a sparse infinite
subset of $\nn$.
\begin{lem}\label{old defn yields new}
Let $k\in\nn$, $(x_s)_{s\in [\nn]^k}$ be a $k$-sequence in a Banach
space $X$,  $L\in[\nn]^\infty$, $(E,\|\cdot\|_*)$ be an infinite
dimensional seminormed linear space with Hamel basis $(e_n)_{n}$ and
 $(\delta_n)_{n}$ be a null sequence of positive reals such that
\begin{equation}\label{ewq}\Bigg{|}\Big{\|}\sum_{j=1}^l a_j
x_{t_j}\Big{\|}-\Big{\|}\sum_{j=1}^l a_j e_j\Big{\|}_*
\Bigg{|}\leq\delta_l\end{equation} for every $l\in\nn$, every
$(t_j)_{j=1}^l\in\textit{Plm}_l([ L]^k)$ with $t_1(1)\geq L(l)$
and every choice of $a_1,...,a_l\in[-1,1]$.
  Then there exists
$M\in[L]^\infty$ such that $(x_s)_{s\in [M]^k}$ generates
$(e_n)_{n}$ as a $k$-spreading model with respect to $(\delta_n)_n$.
\end{lem}
\begin{proof}
We choose $M\in [L]^\infty$ such  that for every $l\in\nn$ there
exist at least $l-1$ elements of $L$ between $M(l)$ and $M(l+1)$.
Then notice that for every $m,l\in\nn$ with $m\leq l$ and every
$(s_j)_{j=1}^m\in\text{Plm}_m([M]^k)$ with $s_j(1)\geq M(l)$, there
exists $(t_j)_{j=1}^l\in\text{Plm}_l([L]^k)$ with $s_j=t_j$ for all
$1\leq j\leq m$. This observation and (\ref{ewq}) easily yield that
for every $m,l\in \nn$, with $m\leq l$, every
$(s_j)_{j=1}^m\in\textit{Plm}_m([M]^k)$ with $s_1(1)\geq M(l)$
and every choice of $a_1,...,a_m\in[-1,1]$, we have
\begin{equation}\Bigg{|}\Big{\|}\sum_{j=1}^m a_j
x_{s_j}\Big{\|}-\Big{\|}\sum_{j=1}^m a_j
 e_j\Big{\|}_* \Bigg{|}\leq\delta_l\end{equation}
and the proof is complete.
\end{proof}

\subsection{Existence of $k$-speading models}
In this subsection we will show that every bounded $k$-sequence in a
Banach space $X$ contains a $k$-subsequence which generates a
$k$-spreading model. The proof follows similar lines with the
corresponding one of the classical spreading models.

 For $k\in\nn$ and a $k$-sequence $(x_s)_{s\in
[\nn]^k}$ in a Banach space $X$, we will say that $(x_s)_{s\in
[\nn]^k}$ \emph{admits} $(e_n)_{n}$ as a $k$-spreading model (or
$(e_n)_{n}$ \emph{is} a $k$-spreading model of $(x_s)_{s\in
[\nn]^k}$) if there exists $M\in[\nn]^\infty$ such that the
subsequence $(x_s)_{s\in [M]^k}$ generates $(e_n)_{n}$ as a
$k$-spreading model.
  A $k$-sequence $(x_s)_{s\in
[\nn]^k}$ in $X$ will be called \textit{bounded} (resp. \textit{seminormalized})  if there exists $C>0$ (resp. $0<c\leq C$)
 such that $\|x_s\|\leq C$ (resp. $c\leq\|x_s\|\leq C$), for every $s\in [\nn]^k$.

\begin{thm}
For all $k\in \nn$,  every bounded $k$-sequence in a Banach space
$X$ admits a $k$-spreading model.
\end{thm}
\begin{proof} Let $X$ be  a Banach space and $k\in\nn$, $(x_s)_{s\in [\nn]^k}$ be  a
bounded $k$-sequence in $X$. We divide the proof into four  steps.

\textbf{Step 1.} Let  $l\in\nn$, $N\in[\nn]^\infty$ and $\delta>0$. Then
there exists $L\in[N]^\infty$ such that
\[\Bigg{|}\Big{\|}\sum_{j=1}^la_jx_{t_j}\Big{\|}-\Big{\|}\sum_{j=1}^la_jx_{s_j}\Big{\|}\Bigg{|}\leq\delta\]
for every $(t_j)_{j=1}^l, (s_j)_{j=1}^l
\in\textit{Plm}_l([L]^k)$ and $a_1,...,a_l\in[-1,1]$.
\begin{proof}[Proof of Step 1:]
Let $(\textbf{a}_i)_{i=1}^{n_0}$ be a $\frac{\delta}{3l}-$net of the
unit ball of $\big(\rr^l, \|\cdot\|_\infty\big)$. We set $N_0=N$. By
a finite induction on $1\leq i\leq n_0$, we construct a decreasing
sequence $N_0\supseteq N_1\supseteq\ldots\supseteq N_{n_0}$ as
follows. Suppose that $N_0,\ldots,N_{i-1}$ have been constructed.
Let $\textbf{a}_i=(a_j^i)_{j=1}^l$ and
$g_i:\textit{Plm}_l([N_{i-1}]^k)\to[0,lC]$ defined by
$g_i\big((s_j)_{j=1}^l\big)=\|\sum_{j=1}^l a_j^i x_{s_j}\|$. We
partition  the interval $[0,lC]$ into disjoint  intervals of length
$\frac{\delta}{3}$ and applying Theorem \ref{ramseyforplegma} we
find  $N_i\in[N_{i-1}]^\infty$ such that for every
$(t_j)_{j=1}^l,(s_j)_{j=1}^l\in\textit{Plm}_l( [N_i]^k)$, we
have $|g_i((t_j)_{j=1}^l)-g_i((s_j)_{j=1}^l)|<\frac{\delta}{3}$.
Proceeding in this way  we conclude that for every
  $(s_j)_{j=1}^l,(t_j)_{j=1}^l\in\textit{Plm}_l([ N_{n_0}]^k)$ and
  $1\leq i\leq n_0$, we have that
  $\Big{|}\|\sum_{j=1}^la_j^ix_{t_j}\|-\|\sum_{j=1}^la_j^ix_{s_j}\|\Big{|}\leq\frac{\delta}{3}$.
Since $(\textbf{a}_i)_{i=1}^{n_0}$ is a $\frac{\delta}{3}-$net of
the unit ball of $(\rr^l, \|\cdot\|_\infty)$ it is easy to see that
$L=N_{n_0}$ is as desired.
\end{proof}

\textbf{Step 2.}  Let  $(\delta_n)_n$ be a  null sequence of positive real
numbers. Then there exists $M\in[\nn]^\infty$ such that for every
$m\leq l$, every $(t_j)_{j=1}^m, (s_j)_{j=1}^m
\in\textit{Plm}_m([M]^k)$ with $s_1(1),t_1(1)\geq M(l)$ and
$a_1,...,a_m\in[-1,1]$, we have
\begin{equation}\label{equo}\Bigg{|}\Big\|\sum_{j=1}^m a_jx_{t_j}\Big\|-\Big\|\sum_{j=1}^m
a_jx_{s_j}\Big\|\Bigg{|}\leq\delta_l\end{equation}

\begin{proof}[Proof of Step 2:] By Step 1 and a standard
diagonalization we easily obtain an  $L\in [\nn]^\infty$ satisfying
$\Big{|}\|\sum_{j=1}^l a_jx_{t_j}\|-\|\sum_{j=1}^l
a_jx_{s_j}\|\Big{|}\leq\delta_l$, for every $ l\in\nn$, every
$(t_j)_{j=1}^l, (s_j)_{j=1}^l \in\textit{Plm}_l([L]^k)$ with
$s_1(1),t_1(1)\geq L(l)$ and $a_1,...,a_l\in[-1,1]$. By Lemma
\ref{old defn yields new}, there exists  $M\in [L]^\infty$
satisfying (\ref{equo}).
\end{proof}

\textbf{Step 3.} Let $M\in[\nn]^\infty$ be the resulting from Step 2
infinite subset of $\nn$. Also let  $l\in\nn$ and
$a_1,...,a_l\in\rr$. Then for  every sequence
$\big((s_j^n)_{j=1}^l\big)_{n}$ with
$(s_j^n)_{j=1}^l\in\text{Plm}_l([M]^k)$, for all $n\in\nn$ and
$\lim s^n_1(1)=+\infty$, the sequence
$(\|\sum_{j=1}^la_jx_{s_j}^n\|)_{n}$ is a Cauchy sequence in
$[0,+\infty)$. Moreover, $\lim_n\|\sum_{j=1}^la_jx_{s_j}^n\|$ is
independent from the choice of the sequence
$((s_j^n)_{j=1}^l)_{n}$.
\begin{proof}[Proof of Step 3:] It is straightforward by
Step 2.
\end{proof}

\textbf{Step 4.} Let $(e_n)_n$ be the natural Hamel basis of $c_{00}(\nn)$.
For every  $l\in\nn$ and $a_1,...,a_l\in\rr$, we define
\[\Big\|\sum_{j=1}^l
a_je_j\Big\|_*=\lim_n\|\sum_{j=1}^la_jx_{s_j}^n\|\] where for every
$n\in\nn$, $(s_j^n)_{j=1}^l\in\text{Plm}_l([M]^k)$ and $\lim
s^n_1(1)=+\infty$. Then $\|\cdot\|_*$ is a seminorm on $c_{00}(\nn)$
under which the natural Hamel basis $(e_n)_n$  is a spreading
sequence. Moreover  for all $ m\leq l$, $a_1,\ldots, a_m\in[-1,1]$
and $(s_j)_{j=1}^m\in\textit{Plm}_m([M]^k)$ with $s_1(1)\geq
M(l)$, we have $\Big{|}\|\sum_{j=1}^m a_jx_{s_j}\|-\|\sum_{j=1}^m
a_je_j\|_*\Big{|}\leq\delta_l$.
\begin{proof}[Proof of Step 4:] It follows easily by Steps 2 and 3.
\end{proof}
By Step 4, we have that $(x_s)_{s\in [M]^{k}}$ generates $(e_n)_n$
as a $k$-spreading model and the proof is complete.
\end{proof}

\subsection{The increasing hierarchy of $k$-spreading models}
In this subsection we will show that the $k$-spreading models of a
Banach space $X$ form an increasing hierarchy.

We start with the  following lemma which is an easy consequence of
Remark \ref{rem3456}.
\begin{lem}\label{propxi}
Let $k_1,k_2\in\nn$  with $1\leq k_1<k_2$. Let $X$ be a Banach space
and $(w_t)_{t\in [\nn]^{k_1}}$ be a $k_1$-sequence in $X$. Let
$(x_s)_{s\in [\nn]^{k_2}}$ be the $k_2$-sequence in $X$  defined by
$x_s=w_{s|k_1}$, for every $s\in  [\nn]^{k_2}$. Then $(w_t)_{t\in
[\nn]^{k_1}}$ and $(x_s)_{s\in [\nn]^{k_2}}$ admit the same
$k$-spreading models.
\end{lem}
For a subset $A$ of $X$ we will say that $A$ \emph{admits} $(e_n)_n$
\emph{as a} $k$-spreading model (or  $(e_n)_{n}$ \emph{is} a
$k$-spreading model of $A$)  if there exists a $k$-sequence
$(x_s)_{s\in [\nn]^k}$ in $A$ which admits $(e_n)_{n}$ as a
$k$-spreading model.
\begin{notation}
  Let $X$ be a Banach space, $A\subseteq X$ and $k\in\nn$. The set of all $k$-spreading models of $A$ will be denoted by $\mathcal{SM}_k(A)$.
\end{notation}
 By Lemma \ref{propxi}, we easily
obtain the following.

\begin{cor}\label{incresspr}
Let $X$ be a Banach space and $A\subseteq X$. Then for all
$k_1,k_2\in\nn$ with $k_1< k_2$, we have
$\mathcal{SM}_{k_1}(A)\subseteq \mathcal{SM}_{k_2}(A)$,
\end{cor}

In Section \ref{s12}, for each $k\in\nn$, we construct a Banach
space $\mathfrak{X}_{k+1}$ such that
$\mathcal{SM}_{k}(\mathfrak{X}_{k+1})\subsetneqq\mathcal{SM}_{k+1}(\mathfrak{X}_{k+1})$.  Here, we
present a much simpler example of a space $X$ and a proper subset $A$ of $X$ satisfying $\mathcal{SM}_{k}(A)\subsetneqq\mathcal{SM}_{k+1}(A)$.

\begin{examp}
\label{example} Let $(e_n)_{n}$ be a normalized spreading and
$1$-unconditional sequence in a Banach space $(E,\|\cdot\|)$ which
is not equivalent to the usual basis of $c_0$. Let $k\in\nn$ and
$(x_s)_{s\in [\nn]^{k+1}}$ be the natural Hamel basis of
$c_{00}([\nn]^{k+1})$. For $x\in c_{00}([\nn]^{k+1})$ we define
\[\|x\|_{k+1}=\sup\Big{\{}\Big\|\sum_{i=1}^lx(s_i)e_i\Big\|:
l\in\nn,(s_i)_{i=1}^l\in\textit{Plm}_l([\nn]^{k+1})\text{ and
}s_1(1)\geq l\Big{\}}\] We set
$X=\overline{(c_{00}([\nn]^{k+1}),\|\cdot\|_{k+1})}$ and
$A=\{x_s:s\in [\nn]^{k+1}\}$. It is easy to see that the sequence
$(e_n)_{n}$ is generated by
$(x_s)_{s\in[\nn]^{k+1}}$ as a $(k+1)$-spreading model and thus it belongs to
$\mathcal{SM}_{k+1}(A)$. We shall show that for every
$(\tilde{e}_n)_{n}\in\mathcal{SM}_k(A)$, either $(\tilde{e}_n)_{n}$
is a trivial spreading sequence or it is isometric to the usual
basis of $c_0$.  Therefore, there is no sequence in
$\mathcal{SM}_{k}(A)$ equivalent to $(e_n)_n$.

Indeed, let  $(\tilde{e}_n)_{n}\in\mathcal{SM}_k(A)$. By Proposition
\ref{remark on the definition of spreading model}, we may assume
that there exists
 a $k$-sequence in $A$, $(y_t)_{t\in [\nn]^k}$ which generates
 $(\tilde{e}_n)_{n}$ as a
$k$-spreading model. Let $\varphi:[\nn]^k\to [\nn]^{k+1}$ such that
$y_t=x_{\varphi(t)}$, for all $t\in [\nn]^{k}$. By Proposition
\ref{lemma making a hereditary nonconstant function, nonconstant on
plegma pairs}, there exists $M\in [\nn]^\infty$ such that either
$\varphi$ is constant on $[M]^{k}$ or for every plegma pair
$(t_1,t_2)$ in $[M]^{k}$, $\varphi(t_1)\neq\varphi(t_2)$. By
Proposition \ref{remark on the definition of spreading model}, we
have that $(y_t)_{t\in [M]^k}$ also generates $(\tilde{e}_n)_{n}$ as
a $k$-spreading model.

If $\varphi$ is constant on $[M]^k$ then $(\tilde{e}_n)_{n}$ is a
trivial sequence. Otherwise,  by Theorem \ref{non plegma preserving
maps}, there exists $L\in[M]^\infty$ such that for every plegma pair
$(t_1,t_2)$ in $[L]^k$ neither $(\varphi(t_1),\varphi(t_2))$, nor
$(\varphi(t_2),\varphi(t_1))$ is a plegma pair in $[\nn]^{k+1}$.
Therefore, for every $(t_j)_{j=1}^m\in\textit{Plm}([L]^k)$ and
$(s_j)_{j=1}^l\in\textit{Plm}([\nn]^{k+1})$ there is  at most
one $j\in\{1,...,m\}$ and at most one $i\in\{1,...,l\}$ with
$\varphi(t_j)=s_i$. This observation  and the definition of the norm
$\|\cdot\|_{k+1}$, easily implies that
\begin{equation}\label{eqwert}\Big\|\sum_{j=1}^ma_jy_{t_j}\Big\|_{k+1}
=\Big\|\sum_{j=1}^ma_jx_{\varphi(t_j)}\Big\|_{k+1}=\max_{1\leq
j\leq m}|a_j|\end{equation}
 for all $m\in\nn$, $a_1,\ldots,a_m\in\rr$ and
$(t_j)_{j=1}^m\in\textit{Plm}([L]^k)$. Since $L\in [M]^\infty$,
we have that $(\tilde{e}_n)_{n}$ is generated by $(y_t)_{t\in[L]^k}$
and  by (\ref{eqwert}), the sequence $(\tilde{e}_n)_{n}$ is
isometric to the usual basis of $c_0$.\end{examp}

\section{Topological properties of $k$-sequences}
This section is devoted to the study of the $k$-sequences in a
topological space. We define the convergence of the $k$-sequences in
a topological space and we introduce the notion of the subordinated
$k$-sequences.
\subsection{Convergence of $k$-sequences in topological spaces}
We start with the following  natural extension of the notion of
convergence of sequences in topological spaces.
\begin{defn}\label{defn convergence of f-sequences}
Let $(X,\ttt)$ be a topological space, $k\in\nn$ and
$(x_s)_{s\in[\nn]^k}$ be a $k$-sequence in $X$. Also let
$M\in[\nn]^\infty$ and $x_0\in X$. We will say that
$(x_s)_{s\in[M]^k}$ converges to $x_0$ if for every $U\in\ttt$
with $x_0\in U$ there exists $m\in \nn$ such that for every $s\in
[M]^k$ with $ s(1)\geq M(m)$ we have that $x_s\in U$.
\end{defn}
It is straightforward that if a $k$-subsequence
$(x_s)_{s\in[M]^k}$ in a topological space is convergent to some
$x_0\in X$, then every further $k$-subsequence of $(x_s)_{s\in
[M]^k}$  is also convergent to $x_0$.  Moreover, every continuous
map between two topological spaces preserves the convergence of
$k$-sequences, i.e. if $\phi:(X_1,\mathcal{T}_1)\to
(X_2,\mathcal{T}_2)$ is continuous and $(x_s)_{s\in[M]^k}$
converges to $x_0\in X_1$, then $(\phi(x_s))_{s\in[M]^k}$
converges to $\phi(x_0)\in X_2$.

However, for $k\geq 2$, there are some differences with the
ordinary convergent sequences
in topological spaces. For instance it is  easy to see that for
$k\geq 2$, the convergence of a $k$-sequence $(x_s)_{s\in[M]^k}$
to some $x_0\in X$, does not in general imply that the set
$\{x_s:s\in[M]^k\}$ is relatively compact.

\subsection{Subordinated $k$-sequences}
In this subsection we introduce the definition of the subordinated
$k$-sequences in a topological space. First, recall that the
powerset of $\nn$ is naturally identified with  $\{0,1\}^\nn$. In
this way, for all $k\in\nn$ and  $M\in [\nn]^\infty$, the set
$[M]^{\leq k}$ becomes a compact metric space containing $[M]^k$ as
a dense subspace. Moreover, notice that an element $s\in [M]^{\leq k}$ is
isolated in $[M]^{\leq k}$ if and only if $s\in [M]^k$.
\begin{defn}\label{defn subordinating}
Let $(X,\ttt)$ be a topological space, $k\in\nn$,
$(x_s)_{s\in[\nn]^k}$ be a $k$-sequence in $X$ and
$M\in[\nn]^\infty$. We say that $(x_s)_{s\in[M]^k}$ is
subordinated $($with respect to $(X,\mathcal{T})$$)$ if there exists a
continuous map $\widehat{\varphi}:[M]^{\leq k}\to (X,\mathcal{T})$
such that $\widehat{\varphi}(s)=x_s$, for all $s\in[M]^k$.
\end{defn}

\begin{rem}\label{remark on ff subordinated}
If $(x_s)_{s\in[M]^k}$ is subordinated, then there exists a unique
continuous map $\widehat{\varphi}:[M]^{\leq k}\to (X,\mathcal{T})$
witnessing this. Indeed, this is a consequence of the fact that
$[M]^k$ is dense in $[M]^{\leq k}$. Also,
$\overline{\{x_s:s\in[M]^k\}}=\widehat{\varphi}\big([M]^{\leq
k}\big),$ where $\overline{\{x_s:s\in[M]^k\}}$ is the closure of
$\{x_s : s\in[M]^k\}$ in $X$ with respect to $\mathcal{T}$.
Therefore, $\overline{\{x_s:s\in[M]^k\}}$ is a countable compact
metrizable subspace of $(X,\ttt)$ with Cantor-Bendixson index at
most $k+1$. Also notice that if $(x_s)_{s\in[M]^k}$ is subordinated
then $(x_s)_{s\in[L]^k}$ is also subordinated, for every $L\in
[M]^\infty$.
\end{rem}

\begin{prop}\label{subordinating yields convergence}
Let $(X,\ttt)$ be a topological space, $k\in\nn$,
$(x_s)_{s\in[\nn]^k}$ be a $k$-sequence in $X$ and
$M\in[\nn]^\infty$. Suppose that $(x_s)_{s\in[M]^k}$ is
subordinated and let $\widehat{\varphi}:[M]^{\leq k}\to
(X,\mathcal{T})$ be the continuous map witnessing this. Then
$(x_s)_{s\in[M]^k}$ is convergent to
$\widehat{\varphi}(\emptyset)$.
\end{prop}
\begin{proof}
Let $(y_s)_{s\in [M]^k}$ be the $k$-sequence  in $[M]^k$, with
 $y_s=s$, for all
$s\in [M]^k$.  Notice that  $(y_s)_{s\in [M]^k}$  converges to the
empty set and since $\widehat{\varphi}:[M]^{\leq k}\to
(X,\mathcal{T})$ is continuous, we have that
$\big(\widehat{\varphi}(y_s)\big)_{s\in [M]^k}$ converges to
$\widehat{\varphi}(\emptyset)$. Since
$\widehat{\varphi}(y_s)=\widehat{\varphi}(s)=x_s$, for all $s\in
[M]^k$, we conclude that  $(x_s)_{s\in[M]^k}$ is convergent to
$\widehat{\varphi}(\emptyset)$.
\end{proof}

\begin{prop}\label{Create subordinated}
Let $(X,\mathcal{T})$ be a topological space,  $k\in\nn$ and
$(x_s)_{s\in[\nn]^k}$ be a $k$-sequence in $X$.  Then for every
$N\in[\nn]^\infty$ such that $\overline{\{x_s:\;s\in[N]^k\}}$ is a
compact metrizable subspace of $(X,\ttt)$  there exists
$M\in[N]^\infty$ such that $(x_s)_{s\in[M]^k}$ is  subordinated.
\end{prop}
\begin{proof}
The proposition obviously holds  for $k=1$, since in this case,
subordinated and convergent sequences coincide. We proceed by
induction on $k\in\nn$. Assume that Proposition \ref{Create
subordinated} holds for some $k\in\nn$ and let  $(x_s)_{s\in
[N]^{k+1}}$ be a $(k+1)$-sequence in $X$. Let $N\in [\nn]^\infty$
such that $\overline{\{x_s:\;s\in[N]^{k+1}\}}$ is a compact
metrizable subspace of $(X,\ttt)$. We also fix a compatible metric
$d$ of $\overline{\{x_s:\;s\in[N]^{k+1}\}}$.

Inductively we choose a strictly  increasing sequence $(l_n)_n$ in
$\nn$, a decreasing sequence $(L_n)_{n}$ of infinite subsets of $N$
and a $k$-sequence $(x_s)_{s\in[L]^k}$ in $X$, where $L=\{l_n:
n\in\nn\}$ such that for every $n\in\nn$, the following are
satisfied.
\begin{enumerate}
\item[(i)] $l_n<\min L_n$. \item[(ii)]  For every $l\in L_n$ and
every $t\in[\{l_1,...,l_n\}]^k $, $(x_{t\cup\{l\}})_{l\in
L_n}\to x_t$ and in addition if $\max t=l_n$, then
$d(x_{t\cup\{l\}},x_t)<\frac{1}{n}$.
\end{enumerate}
We omit the construction since it is straightforward.  By the
inductive assumption there exists  $M\in[L]^\infty$ such that
$(x_t)_{t\in[M]^k}$ is subordinated.  If $\widehat{\psi}:[M]^{\leq
k}\to X$ is  the continuous map witnessing this then  we extend
$\widehat{\psi}$ to the map $\widehat{\varphi}:[M]^{\leq k+1}\to X$,
by setting $\widehat{\varphi}(s)=x_s$, for every $s\in[M]^{k+1}$.
Using condition (ii), we easily show that
 $\widehat{\varphi}$ is continuous and therefore
$(x_s)_{s\in[M]^{k+1}}$ is subordinated.
\end{proof}
\begin{rem} By Propositions \ref{subordinating yields convergence}
and  \ref{Create subordinated}, we have that  every $k$-sequence in
a compact metrizable space  contains a convergent $k$-subsequence.
\end{rem}

\section{Weakly relatively compact $k$-sequences in Banach spaces}
It is  well known  that for every   sequence $(x_n)_n$ in a weakly
compact subset of a Banach space $X$ there exists $M\in \nn$ such
that the subsequence $(x_n)_{n\in M}$ is weakly convergent to some $x_0\in X$. Moreover, if in addition $X$ has  a
Schauder basis then we may pass to a further subsequence
$(x_n)_{n\in L}$ which  is approximated by a sequence of the form
$(\widetilde{x}_n)_{n\in L}$ such that $(\widetilde{x}_n)_{n\in L}$ also weakly converges to $x_0$ and $(\widetilde{x}_n-x_0)_{n\in L}$ is a block
sequence of $X$.  The main aim of this section is to show that, for every $k\geq2$, the $k$-sequences in Banach spaces satisfy similar properties.
 \begin{defn} A $k$-sequence
$(x_s)_{s\in[\nn]^k}$, of a Banach space  $X$ will be called weakly
relatively compact if $\overline{\{x_s: s\in[\nn]^k\}}^{w}$ is a
weakly compact subset of $X$.
\end{defn}
Since the weak topology on every separable weakly  compact subset of
a Banach space is metrizable, by Propositions \ref{subordinating
yields convergence} and \ref{Create subordinated} we have the
following.
\begin{prop}\label{cor for subordinating}
Let $X$ be a Banach space and  $k\in\nn$. Then we have the following.
\begin{enumerate}
 \item[(i)] Every subordinated $k$-sequence in $(X,w)$  is weakly
convergent. \item[(ii)]  Every  weakly relatively compact
$k$-sequence   in $X$  contains a subordinated $k$-subsequence.
\end{enumerate}
\end{prop}
To describe the regularity properties of weakly relatively compact
$k$-sequences in a Banach space $X$  with Schauder basis we will
need the next two  definitions. The first is a natural extension of
the notion  of   block  (resp. disjointly  supported) sequences of
$X$.
\begin{defn}
  Let $X$ be a Banach space with a Schauder basis and $k\in\nn$.
  Let also $(x_s)_{s\in[\nn]^k}$ be a $k$-sequence in $X$ and
  $M\in[\nn]^\infty$. We will say that the $k$-subsequence
  $(x_s)_{s\in[M]^k}$ is plegma block (resp. plegma disjointly
  supported) if for all plegma pairs $(s_1,s_2)$ in $[M]^k$ we
  have $\text{supp}(x_{s_1})<\text{supp}(x_{s_2})$ (resp.
  $\text{supp}(x_{s_1})\cap\text{supp}(x_{s_2})=\emptyset$).
\end{defn}
\begin{defn}\label{Def of plegma supported}
Let $X$ a Banach space with a Schauder basis,  $k\in\nn$ and
$(x_s)_{s\in[\nn]^k}$ be  a $k$-sequence in $X$.  Also let $L\in
[\nn]^\infty$ and $(y_t)_{t\in[L]^{\leq k}}$ be a family of vectors
in $X$. We will say that $(y_t)_{t\in[L]^{\leq k}}$ is a canonical
tree decomposition of $(x_s)_{s\in[L]^k}$ (or $(x_s)_{s\in[L]^k}$
admits $(y_t)_{t\in[L]^{\leq k}}$ as a canonical  tree
decomposition) if the following are satisfied.
\begin{enumerate}
\item[(i)] For every $s\in[L]^k$,
$\displaystyle x_s=\sum_{j=0}^{k}y_{s|j}=
y_\emptyset+\sum_{j=1}^{k}y_{s|j}$.
\item[(ii)] For every $t\in[L]^{\leq k}\setminus\{\emptyset\}$,
$\text{supp}(y_{t})$ is finite. \item[(iii)] For every
$s\in[L]^k$ and $1\leq j_1<j_2\leq k$, $\text{supp}(y_{s|j_1})
<\text{supp}(y_{s|j_2})$.
\item[(iv)] For every $(s_1,s_2)\in\textit{Plm}_2([L]^k)$
 and  $1\leq j_1\leq j_2\leq k$, we
have\[\text{supp}(y_{s_1|j_1}) <\text{supp}(y_{s_2|j_2})\]
\item[(v)] For every $(s_1,s_2)\in\textit{Plm}_2([L]^k)$ and  $1\leq j_1<j_2\leq k$,
we have \[\text{supp}(y_{s_2|j_1}) <\text{supp}(y_{s_1|j_2})\]
\end{enumerate}
\end{defn}
The next proposition gathers some basic properties of the
$k$-sequences which admit canonical tree decomposition. Its proof is straightforward.
\begin{prop}\label{trocan} Let $X$ a Banach space with a Schauder basis,
$k\in\nn$, $(x_s)_{s\in[\nn]^k}$ be  a $k$-sequence in $X$ and $L\in
[\nn]^\infty$. Assume that  $(x_s)_{s\in[L]^k}$ admits
$(y_t)_{t\in[L]^{\leq k}}$ as a canonical tree decomposition. Then
the following are satisfied.
\begin{enumerate}
\item[(i)] For every $N\in[L]^\infty$, the $k$-subsequence
$(x_s)_{s\in[N]^k}$ admits $(y_t)_{t\in[N]^{\leq k}}$ as a
canonical tree decomposition. \item[(ii)] For every $s\in [L]^k$,
the sequence $(y_{s|j})_{j=1}^k$ is a block sequence in $X$.
\item[(iii)] For every  $1\leq j\leq k$, the sequence
$(y_{s|j})_{s\in [L]^k}$ is a plegma block $k$-sequence in $X$.
\item[(iv)] Setting $x'_s=x_s-y_\emptyset$, for all $s\in[L]^k$,
$y'_\emptyset=0$ and $y'_t=y_t$, for all $t\in[L]^{\leq k}$ with
$t\neq\emptyset$, we have that the $k$-subsequence
$(x'_s)_{s\in[L]^k}$ is plegma disjointly supported and admits
$(y'_t)_{t\in[L]^{\leq k}}$ as a canonical tree decomposition.
\item[(v)] For every $j\in\{1,..,k\}$ and
$(s_i)_{i=1}^n\in\text{Plm}_n([L]^k)$, if $I$ is the interval of
$\nn$ with $\min I=\min\text{supp}(y_{s_1|j})$ and $\max I= \max
\text{supp}(y_{s_n|j})$, then for every $1\leq i\leq n$,
$I(x_{s_i}-y_\emptyset)=y_{s_i|j}$.
\end{enumerate}
\end{prop}
The following is the main result of this section.
\begin{thm}\label{canonical tree}
Let $X$ be a Banach space with  Schauder basis, $k\in\nn$,
$(x_s)_{s\in[\nn]^k}$ be a  $k$-sequence in $X$ and $(\ee_n)_{n}$ be
a null sequence of positive reals. Assume that for some  $M\in
[\nn]^\infty$, $(x_s)_{s\in[M]^k}$ is subordinated with respect to
the weak topology of $X$ and let $x_0$ be the weak limit of
$(x_s)_{s\in[M]^k}$. Then there exist $L\in[M]^\infty$ and a
$k$-subsequence $(\widetilde{x}_s)_{s\in[L]^{k}}$ in $X$ satisfying
the following.
\begin{enumerate}
\item[(i)] $(\widetilde{x}_s)_{s\in[L]^k}$ admits a canonical tree  decomposition
$(y_t)_{t\in[L]^{\leq k}}$ with $y_\emptyset=x_0$.
\item[(ii)] For every
$s\in[L]^k$, $\|x_s-\widetilde{x}_s\|<\ee_n$, where $\min
s=L(n)$.
\item[(iii)]  $(\widetilde{x}_s)_{s\in[L]^k}$ is subordinated with
respect to the weak topology of $X$. Moreover $x_0$ is the weak
limit of $(\widetilde{x}_s)_{s\in[L]^k}$.
\end{enumerate}
\end{thm}

\begin{proof}
 Without loss of generality, we may assume that $(\ee_n)_n$ is decreasing. We will first define a
family $(y_t)_{t\in [M]^k}$ of finitely supported vectors in $X$ as
follows. Let $\widehat{\varphi}:[M]^{\leq k}\to (X,w)$ be the
continuous map witnessing that $(x_s)_{s\in[M]^k}$ is subordinated.
For $t=\emptyset$, we set $y_\emptyset=
\widehat{\varphi}(\emptyset)=x_0$. For
$t\in[M]^{\leq k}\setminus\{\emptyset\}$, let
$w_{t}=\widehat{\varphi}(t)-\widehat{\varphi}(t\setminus\{\max
t\})$. Notice that the sequence $(w_{t\cup\{m\}})_{m\in M}$ is
weakly null, for all $t\in[M]^{< k}$. Hence, by a sliding hump
argument, we may choose a family $\big\{I_t: t\in [M]^{\leq
k}\setminus\{\emptyset\}\big\}$ of finite intervals of $\nn$
satisfying the following properties.
\begin{enumerate}
\item[(P1)] For every $t\in [M]^{\leq k}$, with $t\neq \emptyset$, we
have that  $\|w_t-y_t\|=\|I_t^c(w_t)\|<\ee_n/k$, where
$M(n)=\max t$.
\item[(P2)] For every $t\in [M]^{< k}$, $ \min
I_{t\cup\{m\}}\stackrel{ m\in M}{\longrightarrow}\infty$.
\end{enumerate}
Now for every $t\in [M]^{\leq k}\setminus \{\emptyset\}$, we set
${y}_t=I_t(w_t)$ and the definition of the family $(y_t)_{t\in
[M]^k}$ is completed. Also, for every $s\in[M]^k$, we set
$\widetilde{x}_s=\sum_{t\sqsubseteq s}y_t$.

We claim that there exists $L\in [M]^\infty$ such that $(y_t)_{t\in
[L]^k}$ is a canonical tree decomposition of
$(\widetilde{x}_s)_{s\in [L]^k}$.  Indeed,  using (P2) and Ramsey's
theorem, there exists $M_1\in [M]^\infty$ such that for every
$s\in[M_1]^k$ and $1\leq j_1<j_2\leq k$, $\text{supp}(y_{s|j_1})
<\text{supp}(y_{s|j_2})$. Using again (P2) and Theorem
\ref{ramseyforplegma}, we find $M_2\in [M_1]^\infty$ such that for
every $(s_1,s_2)\in\textit{Plm}_2([M_2]^k)$ and $1\leq j_1\leq
j_2\leq k$, $\text{supp}(y_{s_1|j_1}) <\text{supp}(y_{s_2|j_2})$,
while   for every $1\leq j_1<j_2\leq k$, $\text{supp}(y_{s_2|j_1})
<\text{supp}(y_{s_1|j_2})$. We set $L=M_2$. By the above, we have
that all conditions (i)-(v) of Definition \ref{Def of plegma
supported} are fulfilled and therefore $(y_t)_{t\in [L]^{\leq k}}$
is a canonical tree decomposition of $(\widetilde{x}_s)_{s\in
[L]^k}$ and the proof of the claim is complete.

Notice that $x_s-\widetilde{x}_s=\sum_{j=1}^k (w_s|j-y_s|j)$, for
all $s\in [L]^k$. Hence by (P1) and since $(\ee_n)_n$ is decreasing,
 we get that $\|x_s-\widetilde{x}_s\|\leq \ee_n$, where $L(n)=\min s$.  It
remains to show that $(\widetilde{x}_s)_{s\in[L]^k}$ is
subordinated. To this end, let $\widetilde{\varphi}:[L]^{\leq k}\to
X$ defined by  $\widetilde{\varphi}(t)=\sum_{u\sqsubseteq t}y_u$,
for all $t\in[L]^{\leq k}$. Clearly
$\widetilde{\varphi}(\emptyset)=y_{\emptyset}=\widehat{\varphi}(\emptyset)$
and $\widetilde{x}_s=\widetilde{\varphi}(s)$, for all $s\in[L]^k$.
To show that $\widetilde{\varphi}$ is  continuous let $(t_n)_{n}$ be
a sequence in $[L]^{\leq k}$ and $t\in [L]^{\leq k}$ such that
$(t_n)_{n}$ converges to $t$. Setting $\max t_n= M(k_n)$,  we may
assume that $k_n\to\infty$. Then
\[\begin{split}
  \|(\widehat{\varphi}(t_n)-\widehat{\varphi}(t))-(\widetilde{\varphi}(t_n)-\widetilde{\varphi}(t))\|\leq\sum_{t\sqsubset u\sqsubseteq
  t_n}\|w_u-y_u\|\leq \ee_{k_n}\;\substack{\longrightarrow\\n\to\infty}\;0
\end{split}\]
Since
$\widehat{\varphi}(t_n)\stackrel{w}{\to}\widehat{\varphi}(t)$, we
get that
$\widetilde{\varphi}(t_n)\stackrel{w}{\to}\widetilde{\varphi}(t)$
and the proof is completed.
\end{proof}
\begin{notation}
Let $X$ be a Banach space and $k\in\nn$. By
$\mathcal{SM}_k^{wrc}(X)$ we will denote the set of all spreading
sequences $(e_n)_n$ such that there exists a weakly relatively
compact $k$-sequence  of $X$ which generates $(e_n)_n$ as a
$k$-spreading model. Notice that
$\mathcal{SM}_k^{wrc}(X)=\mathcal{SM}_k(X)$, for every  reflexive
space $X$ and $k\in\nn$.
\end{notation}

\begin{cor}\label{cor canonical tree with spr mod}
Let $X$ be  a Banach space  with  Schauder basis and $k\in\nn$. Then
every $(e_n)_{n}\in\mathcal{SM}_k^{wrc}(X)$ is generated by a
$k$-sequence in $X$ which is subordinated with respect to the weak
topology  and admits a canonical tree decomposition.
\end{cor}
\begin{proof}
 Let  $k\in\nn$ and $(x_s)_{s\in[\nn]^k}$ be a weakly relatively compact $k$-sequence
in $X$ which generates  a $k$-spreading model $(e_n)_{n}$.  By
Proposition \ref{cor for subordinating}, there exists
$M\in[\nn]^\infty$  such that  $(x_s)_{s\in[M]^k}$ is subordinated.
By Theorem \ref{canonical tree}, there exists $L\in [M]^\infty$ and a
subordinated sequence $(\widetilde{x_s})_{s\in[L]^k}$ in $X$ which
admits a canonical tree decomposition such that
$\|x_s-\widetilde{x}_s\|<1/n$, for every $s\in [L]^k$ with $\min
s=L(n)$. Hence there is $N\in [L]^\infty$ such that
$(\widetilde{x_s})_{s\in[N]^k}$ also generates $(e_n)_n$ as a
$k$-spreading model. Setting $z_s=\widetilde{x}_{N(s)}$, for all $s\in[\nn]^k$, we have
that $(z_s)_{s\in [\nn]^k}$ is as desired.
 \end{proof}
\section{Norm properties of spreading models}\label{s5} In this section we provide conditions for $k$-sequences to admit unconditional, singular or trivial spreading models. Our main interest concerns subordinated $k$-sequences with respect to the weak topology.
\subsection{Unconditional spreading models} As is well known every
spreading model generated by a seminormalized weakly null sequence
is an $1$-unconditional spreading sequence. In this subsection we
give an extension of this result for subordinated
seminormalized weakly null $k$-sequences.

\begin{lem}\label{Lemma finding convex  means}
Let $k\in\nn$ and $(x_s)_{s\in[\nn]^k}$ be a $k$-sequence in a
Banach space $X$. Suppose that $(x_s)_{s\in[\nn]^k}$ is
subordinated and let $\widehat{\varphi}:[\nn]^{\leq k}\to (X,w)$
be the continuous map witnessing this. Let $\ee>0$, $M\in
[\nn]^\infty$ and $ n\in\nn$. Then for every $p\in\{1,...,n\}$
there exists a finite subset $G$ of $[M]^k$  such that  the
following are satisfied.
\begin{enumerate}
\item[(i)]  There exists a convex combination $x=\sum_{s\in
G}\mu_sx_s$  of $(x_s)_{s\in G}$ such that
$\|\widehat{\varphi}(\emptyset)-x\|<\ee$.
\item[(ii)]  For every $ 1\leq i\leq n$ with  $i\neq p$, there
exists $s_i\in [M]^k$ such that for every $s_p\in G$, the family
 $(s_i)_{i=1}^n$  is a plegma family in $[M]^k$.
\end{enumerate}
\end{lem}
\begin{proof} For $k=1$, the result follows by Mazur's theorem.
We proceed by induction on $k\in\nn$. Assume  that the lemma  is
true for some $k\in\nn$. We fix a subordinated $(k+1)$-sequence
$(x_s)_{s\in[\nn]^{k+1}}$ in $X$, $M\in[\nn]^\infty$, $ n\in \nn$,
$\ee>0$ and $p\in\{1,...,n\}$.

Let  $(x_t)_{t\in[M]^k}$ defined by $x_t=\widehat{\varphi}(t)$,
for all $t\in [M]^k$. By our inductive assumption, there exists a
finite subset $F$ of $[M]^k$ satisfying the following.
\begin{enumerate}
\item[(a)] There exists a convex combination $\sum_{t\in
F}\mu_tx_t$ of $(x_t)_{t\in F}$ such that
\begin{equation}\label{hg}\Big\|\widehat{\varphi}(\emptyset)-\sum_{t\in
F}\mu_tx_t\Big\|<\ee/2\end{equation} \item[(b)]  For every $ 1\leq
i\leq n$ with  $i\neq p$, there exists $t_i\in [M]^k$ such that
for every $t_p\in F$, $(t_i)_{i=1}^n$ is a plegma family in
$[M]^k$.
\end{enumerate}
For notational simplicity we assume that $1<p<n$ (the proof for
$p\in\{1,n\}$ is similar). Pick $m_1<\ldots<m_{p-1}$ in $M$ with
$t_n(k)< m_1$ and set $s_i=t_i\cup \{m_i\}$, for all
$i=1,\ldots,p-1$. Also let $M'=\{m\in
M:m>m_{p-1}\}$. Since $\widehat{\varphi}$ is continuous, we have
that $(x_{t\cup\{m\}})_{m\in M'}\stackrel{w}{\to}x_t$, for every
$t\in F$. Hence by Mazur's theorem, for every $t\in F$, there
exists a finite subset $G_t$ of $M'$  such that
\begin{equation}\label{hgg}\Big\|x_t-\sum_{m\in
G_t}\mu^t_mx_{t\cup\{m\}}\Big\|<\ee/2
\end{equation}
for some convex combination $\sum_{m\in G_t}\mu^t_m
x_{t\cup\{m\}}$ of $(x_{t\cup\{m\}})_{m\in G_t}$. We set \[G=\{t\cup\{m\}:t\in F\;\text{and}\;m\in
G_t\}\] Finally, pick $m_{p+1}<...<m_n$ in $M$ with
$\max\{m:m\in\bigcup_{t\in F} G_t\}<m_{p+1}$ and let
$s_i=t_i\cup\{m_i\}$, for all $i=p+1,\ldots,n$.

It is easy to check that  every $(s_i)_{i=1}^n$ with $s_p\in G$,
is a plegma family in $[M]^{k+1}$. It remains to show that
condition (i) of the lemma  is also satisfied.  To this end, let
$\mu_s=\mu_t\mu^t_m$, for every $s=t\cup\{m\}\in G$, where $\max t<m$. Notice that
\[\sum_{s\in G}\mu_s=\sum_{t\in F}\mu_t\sum_{m\in G_t}\mu^t_m=\sum_{t\in F}\mu_t=1\]
and therefore $\sum_{s\in G}\mu_s x_s$ is a convex combination of
$(x_s)_{s\in G}$. Moreover, we have
\[\begin{split}
\Big\|\widehat{\varphi}(\emptyset)-\sum_{s\in G}&\mu_sx_s\Big\|=
\Big\|\widehat{\varphi}(\emptyset)-\sum_{t\in F}\mu_t\sum_{m\in G_t}\mu^t_m x_{t\cup\{m\}}\Big\|\\
&\leq\Big\|\widehat{\varphi}(\emptyset)-\sum_{t\in
G'}\mu'_tx_t\Big\|+ \sum_{t\in F}\mu_t\cdot\Big\|x_t-\sum_{m\in
G_t}\mu^t_mx_{t\cup\{m\}}\Big\|\stackrel{(\ref{hg}),
(\ref{hgg})}{<}\ee
\end{split}\] and the proof is complete.
\end{proof}

\begin{thm}\label{unconditional spreading model}
Let $k\in\nn$ and $(x_s)_{s\in[\nn]^k}$ be a $k$-sequence in a
Banach space $X$. Suppose that $(x_s)_{s\in[\nn]^k}$ is
seminormalized, subordinated (with respect to the weak topology of
$X$) and weakly null. Then every $k$-spreading model of
$(x_s)_{s\in[\nn]^k}$ is $1$-unconditional.
\end{thm}

\begin{proof}
  Let $(e_n)_n$ be a spreading model of $(x_s)_{s\in[\nn]^k}$.
  Lemma \ref{Lemma finding convex  means} and the averaging
  technique used for the proof of the corresponding result in the
  case of the classical spreading models (see \cite{BL} Proposition
  I.5.1) yield that for every $n\in\nn$, $1\leq
p\leq n$, $a_1,\ldots,a_n\in [-1,1]$ and $\ee>0$, we have
\[\Big{\|}\sum_{\substack{i=1\\i\neq
p}}^na_ie_i\Big{\|}_*\leq\Big{\|}\sum_{i=1}^na_ie_i\Big{\|}_*+\varepsilon\]
Since the above inequality holds for every $\ee>0$, we have that
\begin{equation}\label{eq13}\Big{\|}\sum_{\substack{i=1\\i\neq
p}}^na_ie_i\Big{\|}_*\leq\Big{\|}\sum_{i=1}^na_ie_i\Big{\|}_*\end{equation}
for all $n\in\nn$, $1\leq p\leq n$ and $a_1,\ldots,a_n\in [-1,1]$. Since $(x_s)_{s\in[\nn]^k}$ is seminormalized, we have that $\|e_1\|_*>0$. By (\ref{eq13}) we get that $\|e_1-e_2\|_*>0$. By Proposition \ref{sing}, we get that $(e_n)_n$ is non trivial.
An iterated use of (\ref{eq13}) completes the proof.
\end{proof}

We close this subsection by giving an example showing that for
$k\geq 2$ the assumption in Theorem \ref{unconditional spreading
model} that the $k$-sequence is subordinated   is necessary. More
precisely, for every $k\geq 2$,  there exist seminormalized
weakly null  $k$-sequences which generate conditional Schauder
basic spreading models.

\begin{examp} For simplicity we state the example for $k=2$.
Let $(e_n)_n$ be the usual  basis of $c_0$ and
$(x_s)_{s\in[\nn]^2}$ be the  $2$-sequence in $c_0$, defined by
$x_s=\sum_{n=\min s}^{ \max s}e_n$, for all $s\in[\nn]^2$.
Clearly, $(x_s)_{s\in[\nn]^2}$ is a normalized weakly null
$2$-sequence. It is easy to check that  for all $l\in\nn$,
$a_1,\ldots,a_l\in\rr$ and
$(s_j)_{j=1}^l\in\textit{Plm}_l([\nn]^2)$, we have
\[
\Big\|\sum_{j=1}^la_jx_{s_j}\Big\|=\max\Big(\max_{1\leq k\leq
l}\Big|\sum_{j=1}^ka_j\Big|,\max_{1\leq k\leq
\l}\Big|\sum_{j=k}^la_j\Big|\Big)\]
Therefore
 every spreading model of  $(x_s)_{s\in[\nn]^2}$, is
equivalent to the summing basis.
\end{examp}

\subsection{Singular and trivial spreading models}
The results of this subsection concern the $k$-spreading models generated by subordinated $k$-sequences which are not weakly null.
\begin{lem}\label{triv-ell}
Let $X$ be  a Banach space, $k\in\nn$, $(x_s)_{s\in[\nn]^k}$ be a
$k$-sequence in $X$ and $x_0\in X$. Let $x'_s=x_s-x_0$, for all
$s\in[\nn]^k$ and assume that  that $(x_s)_{s\in[\nn]^k}$ and
$(x'_s)_{s\in[\nn]^k}$ generate $k$-spreading models $(e_n)_n$ and
$(\widetilde{e}_n)_n$ respectively. Then the following hold.
\begin{enumerate}
\item[(a)] $\|\sum_{i=1}^na_ie_{_i}\|=\|\sum_{i=1}^na_i\widetilde{e}_{_i}\|$,  for  every $n\in\nn$ and  $a_1,\ldots,a_n\in\rr$ with
$\sum_{i=1}^na_i=0$.
\item[(b)] The sequence $(e_n)_n$ is trivial if and only if $(\widetilde{e}_n)_n$ is trivial.
\item[(c)] The sequence $(e_n)_n$ is equivalent to the usual basis of $\ell^1$
if and only if $(\widetilde{e}_n)_n$ is equivalent to the usual
basis of $\ell^1$.
\end{enumerate}
\end{lem}
\begin{proof}
(a) Notice  that for  every $n\in\nn$, $s_1,...,s_n$ in $[\nn]^k$
and $a_1,\ldots,a_n\in\rr$ with $\sum_{i=1}^na_i=0$, we have
$\sum_{i=1}^na_ix_{s_i}=\sum_{i=1}^na_ix'_{s_i}$. Since $(e_n)_n$
and  $(\widetilde{e}_n)_n$ are generated by $(x_s)_{s\in [\nn]^k}$
and $(x_s)_{s\in [\nn]^k}$ the result follows. \\
(b) It  follows by assertion (a) and Proposition \ref{sing}.\\
(c) We fix $\ee>0$. If $(\widetilde{e}_n)_n$ is not equivalent to
the usual basis of $\ell^1$ then there exist $n\in\nn$ and
$a'_1,\ldots,a'_n\in\rr$ such that $\sum_{i=1}^n|a'_i|=1$ and
$\|\sum_{i=1}^na'_i\widetilde{e}_i\|<\ee$. Setting $a_i=a'_i/2$ and
$a_{n+i}=-a'_i/2$, for all $1\leq i\leq n$, we have
$\sum_{i=1}^{2n}a_i=0$ and therefore,
$\|\sum_{i=1}^{2n}a_ie_i\|=\|\sum_{i=1}^{2n}a_i\widetilde{e}_i\|<\ee$.
Since  $\sum_{i=1}^{2n}|a_i|=1$,  $(e_n)_n$ is also not equivalent
to the usual basis of $\ell^1$.
\end{proof}
\begin{thm}\label{nb}
Let $X$ be  a Banach space, $k\in\nn$ and $(x_s)_{s\in[\nn]^k}$ be a
subordinated $k$-sequence in  $X$. Also let $x'_s=x_s-x_0$, for every
$s\in[\nn]^k$, where $x_0$ is the weak limit of
$(x_s)_{s\in[\nn]^k}$. Assume that for some $M\in [\nn]^\infty$ the $k$-subsequence
$(x_s)_{s\in[M]^k}$ generates a non trivial $k$-spreading model
$(e_n)_n$.  If $x_0\neq 0$, then exactly one of the following holds.
\begin{enumerate}
\item[(i)] The sequence $(e_n)_n$ as well as  every spreading model of $(x'_s)_{s\in [M]^k}$
 is  equivalent to the usual basis of $\ell^1$. \item[(ii)] The
sequence $(e_n)_n$ is singular and if $e_n=e'_n+e$ is its
natural decomposition  then $(e'_n)_n$ is the unique
$k$-spreading model of $(x'_s)_{s\in[M]^k}$ and $\|e\|=\|x_0\|$.
\end{enumerate}
\end{thm}
\begin{proof}
Let  $(\widetilde{e}_n)_n$ be a $k$-spreading model of
$(x'_s)_{s\in[M]^k}$. If  $(e_n)_n$ is equivalent to the usual basis
of $\ell^1$ then by Lemma \ref{triv-ell}, we have that the
same holds for $(\widetilde{e}_n)_n$ and hence  (i) is satisfied.

Assume for the following that  $(e_n)_n$ is not equivalent to the
usual basis of $\ell^1$. Since it is also non trivial,  by
Lemma \ref{triv-ell}, we have that $(\widetilde{e}_n)_n$ is
non trivial and not equivalent to the $\ell^1$-basis. Let $L\in
[M]^\infty$ such that $(x'_s)_{s\in[L]^k}$ generates
$(\widetilde{e}_n)_n$. Since $(\widetilde{e}_n)_n$ is non trivial,
it is easy to see that $(x'_s)_{s\in[L]^k}$ is seminormalized. Also
notice that $(x'_s)_{s\in[M]^k}$ is subordinated and weakly null.
Therefore by Theorem \ref{unconditional spreading model},
$(\widetilde{e}_n)_n$ is 1-unconditional. Moreover, since
$(\widetilde{e}_n)_n$ is not equivalent to the usual basis of
$\ell^1$, by Proposition \ref{equiv forms for 1-subsymmetric weakly
null}, we conclude that $(\widetilde{e}_n)_n$ is  Ces\`aro summable
to zero. Hence we have
\begin{equation}\label{tv} \lim_{n\to\infty}\Big\|\frac{1}{n}\sum_{j=1}^ne_j-\frac{1}{n}\sum_{j=n+1}^{2n}e_j\Big\|
=\lim_{n\to\infty}\Big\|\frac{1}{n}\sum_{j=1}^n\widetilde{e}_j-\frac{1}{n}\sum_{j=n+1}^{2n}\widetilde{e}_j\Big\|=0\end{equation}
Also it is easy to see that
\begin{equation}\label{fd}
\Big\|\frac{1}{n}\sum_{j=1}^ne_j\Big\|\to\|x_0\|>0\end{equation} By
(\ref{tv}) and (\ref{fd}), we get that  $(e_n)_n$ is not Schauder
basic, i.e. it is singular. Let $e_n=e'_n+e$ be the natural
decomposition of $(e_n)_n$. By (\ref{fd}) and the fact that
$(e'_n)_n$ is Ces\`aro summable to zero, we have that
$\|e\|=\|x_0\|$. To complete the proof it remains to show that
$(\widetilde{e}_n)_n$ and $(e'_n)_n$ are isometrically equivalent.
Indeed, we fix $n\in\nn$ and $a_1,\ldots,a_n\in\rr$. For every
$p\in\nn$, let $(s^p_{j})_{j=1}^{n+p}\in\textit{Plm}_p([L]^k)$
such that $s^p_1(1)\geq L(n+p)$.  We also set $a=\sum_{j=1}^na_j$.
Then we have
  \[\begin{split}
    \Big\|\sum_{j=1}^na_je'_j\Big\|&
    =\lim_{p\to\infty}\Big\|\sum_{j=1}^na_je'_j-\frac{a}{p}\sum_{j=n+1}^{n+p}e'_j\Big\|
    =\lim_{p\to\infty}\Big\|\sum_{j=1}^na_je_j-\frac{a}{p}\sum_{j=n+1}^{n+p}e_j\Big\|\\
    &=\lim_{p\to\infty}\Big\|\sum_{j=1}^na_jx_{s_j^p}-\frac{a}{p}\sum_{j=n+1}^{n+p}x_{s_j^p}\Big\|
    =\lim_{p\to\infty}\Big\|\sum_{j=1}^na_jx'_{s_j^p}-\frac{a}{p}\sum_{j=n+1}^{n+p}x'_{s_j^p}\Big\|\\
    &=\lim_{p\to\infty}\Big\|\sum_{j=1}^na_j\widetilde{e}_j-\frac{a}{p}\sum_{j=n+1}^{n+p}\widetilde{e}_j\Big\|
    =\Big\|\sum_{j=1}^na_j\widetilde{e}_j\Big\|
  \end{split}\]
\end{proof}
 By Remark \ref{properties of the natural decomposition}, Proposition \ref{cor for subordinating} and
 Theorems
\ref{unconditional spreading model} and \ref{nb}, we derive the
following.
\begin{cor}\label{cor singular wrc}
Let $X$ be a Banach space, $k\in\nn$ and
$(e_n)_{n}\in\mathcal{SM}^{wrc}_k(X)$ non trivial. Then one of the following
holds.
\begin{enumerate}
\item[(i)] The sequence $(e_n)_n$ is unconditional. \item[(ii)]
The sequence $(e_n)_n$ is singular and if $e_n=e'_n+e$ is the
natural decomposition of $(e_n)_n$ then
$(e'_n)_{n}\in\mathcal{SM}^{wrc}_k(X)$, $(e'_n)_n$ is
unconditional, weakly null and Ces\`aro summable to zero.
Moreover, the spaces generated by $(e_n)_n$ and $(e'_n)_n$ are
isomorphic.
\end{enumerate}
\end{cor}

The next theorem  provides more information concerning the trivial
$k$-spreading models. Since we shall not use this result in the
sequel, we omit its proof.
\begin{thm}\label{Theorem equivalent forms for having norm on the spreading model}
Let $k\in\nn$, $(x_s)_{s\in[\nn]^k}$ be an $k$-sequence in a
Banach space $X$ and $(E,\|\cdot\|_*)$ be an infinite dimensional
seminormed linear space with Hamel basis $(e_n)_n$. Assume that
for some $M\in[\nn]^\infty$, the $k$-subsequence $(x_s)_{s\in
[M]^k}$ generates $(e_n)_n$ as an $k$-spreading model. Then the
following are equivalent:
\begin{enumerate}
\item[(i)] The sequence $(e_n)_n$ is trivial. \item[(ii)] The
seminorm $\|\cdot\|_*$ is not a norm on $E$. \item[(iii)]
$(x_s)_{s\in [M]^k}$ contains a further norm Cauchy
$k$-subsequence, i.e. there exists $L\in[M]^\infty$ such that for
every $\ee>0$ there exists $n_0\in\nn$ satisfying that
$\|x_s-x_t\|<\ee$, for all $s,t\in[L]^k$ with $n_0\leq\min\{\min
s,\min t\}$. \item[(iv)] There exists $x\in X$ such that every
$k$-subsequence of $(x_s)_{s\in[M]^k}$ contains a further
$k$-subsequence convergent to $x$.
\end{enumerate}
\end{thm}

\section{Composition of the spreading models}
In this section we study the composition property of the $k$-spreading models.
Moreover we recall the definition of the $k$-iterated spreading models
 and we investigate their relation with the $k$-spreading models. We start with the following definition.
\begin{defn}
Let  $X$ be a Banach space with a Schauder basis and  $k\in\nn$.
Then a $k$-spreading model  $(e_n)_n$ of $X$ will be called plegma
block generated if there exists a $k$-sequence $(x_s)_{s\in
[\nn]^k}$ which is plegma block and generates $(e_n)_n$ as a
$k$-spreading model.
\end{defn}
\begin{rem}\label{remnj}
By Lemma \ref{propxi}, we easily conclude   that for $1\leq
k_1<k_2$, every plegma block generated $k_1$-spreading model is also
a plegma block $k_2$-spreading model. Thus the plegma block
$k$-spreading models of a Banach space $X$ with a Schauder basis
form an increasing hierarchy.
\end{rem}
\begin{thm}\label{composition thm}
Let  $X$ be a Banach space, $k\in \nn$ and  $(e_n)_n\in
\mathcal{SM}_{k}(X)$ such that $(e_n)_n$ is a Schauder basic
sequence. Let $E$ be the Banach space with Schauder basis the
sequence $(e_n)_n$, $d\in\nn$ and  $(\widetilde{e}_n)_{n}$ be a
plegma block generated $d$-spreading model of $E$. Then
$(\widetilde{e}_n)_{n}\in\mathcal{SM}_{k+d}(X)$.
\end{thm}
\begin{proof}
We fix  a plegma block $d$-sequence $(y_t)_{t\in[\nn]^d}$ in
$E$ which generates $(\widetilde{e}_n)_n$ as a $d$-spreading model
with respect to some null sequence $(\widetilde{\delta}_n)_n$ of
positive reals. By Proposition \ref{remark on the definition of
spreading model}, we may also choose a $k$-sequence
$(x_s)_{s\in[\nn]^k}$ in $X$ which generates $(e_n)_n$  as a
$k$-spreading model with respect to the same sequence
$(\widetilde{\delta}_n)_n$.

Since $(y_t)_{t\in [\nn]^d}$ is finitely supported, setting for
every $t\in[\nn]^d$, $F_t=\text{supp}(y_t)$,
\begin{equation}y_t=\sum_{j=1}^{|F_t|}a^t_{F_t(j)}e_{F_t(j)}^{\;}\end{equation}
 For every $v\in[\nn]^{k+d}$, let $t_v$ (resp.
$s_v$)  be the unique element in $[\nn]^d$ (resp. $[\nn]^k$)
such that $v=t_v\cup s_v$ and $t_v<s_v$. For every $v\in [\nn]^{k+d}$
and  $j\in\{1,...,|F_{t_u}|\}$, we set
\begin{equation}\label{pl}s_j^v=(s_v(1)+j-1,...,s_v(k)+j-1)\end{equation} Notice that
$(s_j^v)_{j=1}^{|F_{t_v}|}$ is a finite sequence  in $[\nn]^k$
with $s_1^v=s_v$.

We define a $(k+d)$-sequence  $(z_v)_{v\in[\nn]^{k+d}}$ in $X$, by setting
\begin{equation}z_v=\sum_{j=1}^{|F_{t_v}|}a^{t_v}_{F_{t_v}(j)}x_{s^v_j}\end{equation}
The proof will be completed once we show the following.

\textbf{Claim 1.}  There exists $M\in[\nn]^\infty$ such that
$(z_v)_{v\in[M]^{k+d}}$ generates $(\widetilde{e}_n)_n$ as a
$(k+d)$-spreading model.

\emph{Proof of Claim 1}: For every $l\in\nn$, we define a family
$\mathcal{A}_l\subseteq \textit{Plm}_l([\nn]^{k+d})$ as follows:
  \[\begin{split}\mathcal{A}_l=\Big{\{} (v_i)_{i=1}^l\in &\textit{Plm}_l([\nn]^{k+d}):
  \; s_1^{v_1}(1)\geq \sum_{i=1}^{l}|F_{t_{v_i}}| \\
  &\text{and}\;\;
  (s_j^{v_1})_{j=1}^{|F_{t_{v_1}|}}\;^\frown \ldots \;^\frown (s_j^{v_l})_{j=1}^{|F_{t_{v_l}|}}\in
  \textit{Plm}_{\sum_{i=1}^{l}|F_{t_{v_i}}|}([\nn]^{k})
   \Big\}\end{split}\]
Using (\ref{pl}),  the fact that for every
$(v_i)_{i=1}^l\in\textit{Plm}_l([\nn]^{k+d})$,
$(s_{v_i})_{i=1}^l\in \textit{Plm}_l([\nn]^{k})$ and that
$s_1^{v_i}=s_{v_i}$, for all $1\leq i\leq l$, it is easy to check
that $\mathcal{A}_l\bigcap \textit{Plm}_l([L]^{k+d})\neq
\emptyset$, for every $l\in\nn$ and $L\in [\nn]^\infty$. Hence, an
iterated use of Theorem \ref{ramseyforplegma}, yields an $L\in
[\nn]^\infty$ such that $(v_i)_{i=1}^l\in\mathcal{A}_l$, for every
$(v_i)_{i=1}^l\in\textit{Plm}_l([L]^{k+d})$, with $v_1(1)\geq
L(l)$.

We fix $l\in\nn$, $(v_i)_{i=1}^l\in\textit{Plm}_l([L]^{k+d})$
with $v_1(1)\geq L(l)$ and $a_1,\ldots,a_l\in[-1,1]$. Notice that
\begin{equation}\begin{split}\label{eq10}
\Bigg| \Big\|\sum_{i=1}^l a_i z_{v_i} \Big\|-\Big\|\sum_{i=1}^l a_i \widetilde{e}_i \Big\| \Bigg|
\leq& \Bigg|\Big\|\sum_{i=1}^l a_i z_{v_i}   \Big\|-\Big\|\sum_{i=1}^l a_i y_{t_{v_i}}  \Big\| \Bigg|\\
&+ \Bigg| \Big\|\sum_{i=1}^l a_i y_{t_{v_i}}  \Big\|- \Big\|\sum_{i=1}^l a_i \widetilde{e}_i \Big\| \Bigg|
\end{split}\end{equation}
Also observe that $(t_{v_i})_{i=1}^l\textit{Plm}_l([L]^{d\widehat{}})$
and $t_{v_1}(1)=v_1(1)\geq L(l)\geq l$. Hence,
\begin{equation}\label{eq11}
\Bigg| \Big\|\sum_{i=1}^l a_i y_{t_{v_i}}  \Big\|-\Big\| \sum_{i=1}^l a_i \widetilde{e}_i  \Big\|\Bigg|<\widetilde{\delta}_l
\end{equation}
Also,   $s_1^{v_1}(1)\geq \sum_{i=1}^l |F_{t_{v_i}}|$ and
$F_{t_{v_1}}<...<F_{t_{v_l}}$. Therefore,
\begin{equation}\begin{split}\label{eq12}
\Bigg|\Big\|\sum_{i=1}^l a_i z_{v_i}   \Big\|-\Big\|\sum_{i=1}^l a_i y_{t_{v_i}}  \Big\| \Bigg|=&
\Bigg|\Big\|\sum_{i=1}^l\sum_{j=1}^{l_{t_{v_i}}} a_ia_{F_{t_{v_i}}(j)}^{t_{v_i}} x_{s^{v_i}_j}  \Big\|\\
&-\Big\|\sum_{i=1}^l\sum_{j=1}^{l_{t_{v_i}}} a_ia_{F_{t_{v_i}}(j)}^{t_{v_i}} e_{F_{t_{v_i}}(j)}^{\;}  \Big\|  \Bigg|<2CK\widetilde{\delta}_l
\end{split}\end{equation}
where  $C$ is  the basis constant of $(e_n)_n$ and
$K=\sup\{\|y_t\|:t\in[L]^k\}$.

By (\ref{eq10}), (\ref{eq11}) and (\ref{eq12}), we obtain that for
every $l\in\nn$, $(v_i)_{i=1}^l\in\textit{Plm}_l([L]^{k})$ with
$v_1(1)\geq L(l)$ and $a_1,\ldots,a_l\in[-1,1]$, we have
\[\Bigg| \Big\|\sum_{i=1}^l a_i z_{v_i} \Big\|-\Big\|\sum_{i=1}^l a_i
\widetilde{e}_i \Big\| \Bigg|<\delta_l\] where $\delta_l=(1+2CK)\widetilde{\delta}_l$.
 By Lemma \ref{old defn
yields new}, there exists $M\in[L]^\infty$, such that
$(z_v)_{v\in[M]^{k}}$ generates $(\widetilde{e}_n)_{n}$ as a
$k$-spreading model and the proof of the claim as well as of Theorem
\ref{composition thm} is complete.
\end{proof}
  \begin{cor}\label{l^p in wrc}
    Let $X$ be a Banach space and $Y$ be either $\ell^p$ for some $p\in[1,\infty)$ or
    $c_0$. Also let  $k\in\nn$,
    $(e_n)_n\in\mathcal{SM}^{wrc}_k(X)$ be non trivial and
    $E$ be the Banach space generated by $(e_n)_n$. Suppose
    that $E$ contains an isomorphic copy of $Y$. Then
    $\mathcal{SM}_{k+1}(X)$ contains a sequence equivalent
    to the usual basis of $Y$.
  \end{cor}
  \begin{proof}
    First assume that that $(e_n)_n$ is Schauder basic. Notice that $E$ contains a block sequence $(y_n)_n$ equivalent to the usual basis of $Y$. It is easy to see that $(y_n)_n$ admits a spreading model $(\widetilde{e}_n)_n$ equivalent to the usual basis of $Y$. By Theorem \ref{composition thm} we have that $(\widetilde{e}_n)_n\in\mathcal{SM}_{k+1}(X)$.

    Assume now that $(e_n)_n$ is not Schauder basic. Since $(e_n)_n$ is non trivial, we have that $(e_n)_n$ is singular. Let $e_n=e'_n+e$ be its natural decomposition and $E'$ the space generated by $(e'_n)_n$.
    By Remark \ref{properties of the natural decomposition} we have that $E$ and $E'$ are isomorphic and therefore $E'$ contains an isomorphic copy of $Y$. By Corollary \ref{cor singular
  wrc} we have that $(e'_n)_n\in\mathcal{SM}_{k+1}(X)$. Since $(e'_n)_n$ is unconditional, the result follows as in the first case.
  \end{proof}

\subsection{The $k$-iterated spreading models}
In this subsection we define the $k$-iterated spreading models of a
Banach space $X$ which although they have not been named, have been
appeared in \cite{BM} and \cite{O-S}. We also study their relation
with the $k$-spreading models.
\begin{defn}
The $k$-iterated spreading models of a Banach space $X$ are
inductively defined  as follows. The 1-iterated are the non trivial
spreading models of $X$. Assume that for some $k\in\nn$ the
$k$-iterated spreading models of $X$ have been defined. Then the
$(k+1)$-iterated spreading models are the non trivial spreading
models of the spaces generated by the $k$-iterated spreading models.
\end{defn}
Notice that the class of the $k$-iterated spading models of a Banach
space $X$ is contained in the one of the $(k+1)$-iterated spreading
models. In the sequel we provide a sufficient condition ensuring
that the $k$-iterated spreading models of a Banach space $X$ are up
to isomorphism contained in $\mathcal{SM}_k(X)$. To this end we need
the following lemma.

\begin{lem}\label{iter_inter}
  Let $X$ be a Banach space and $k\in\nn$. Let $(e^0_n)_n$ be a
  Schauder basic $k$-spreading model of $X$, $E_0$ be the space generated
  by $(e^0_n)_n$, $(e_n)_n$ be a non trivial spreading model of
  $E_0$ and $E$ be the space generated by $(e_n)_n$. If $E_0$ is reflexive then there exists
  an unconditional $(k+1)$-spreading model of $X$ generating a space isomorphic to
  $E$.
\end{lem}
\begin{proof}
Let $(x_n)_n$ be a sequence in $E_0$ generating $(e_n)_n$ as a
spreading model. Since $E_0$ is reflexive, we may assume that
$(x_n)_n$ is weakly convergent to some $x_0\in E_0$. If $x_0=0$,
then $(e_n)_n$ is unconditional and it is generated by a block
sequence in $E_0$, while  if $(e_n)_n$ is equivalent to the usual
basis of $\ell^1$ then $E_0$ contains a block sequence generating an
$\ell^1$ spreading model. Therefore, in both cases the result
follows by Theorem \ref{composition thm}. Assume that $x_0\neq0$ and
$(e_n)_n $ is not equivalent to the usual basis of $\ell^1$. Let
$x'_n=x_n-x_0$, for all $n\in\nn$. By Theorem \ref{nb}, we have that
$(e_n)_n$ is singular and $(e'_n)_n$ is the unique spreading model
of $(x'_n)_n$, where $e_n=e'_n+e$ is the natural decomposition of
$(e_n)_n$. Since $(x'_n)_n$ is weakly null, we have that $(e'_n)_n$
is generated by a block sequence in $E_0$ as a spreading model.
Hence, by Theorem \ref{composition thm}, the sequence $(e'_n)_n$ is
a $(k+1)$-spreading model of $X$. Moreover, by Remark
\ref{properties of the natural decomposition},  $(e'_n)_n$ is
unconditional and the space $E'$ generated by $(e'_n)_n$ is
isomorphic to $E$.
\end{proof}

\begin{prop}\label{iterated_lemma}
Let $X$ be a reflexive space and $k\in\nn$ such that every space
generated by a $k$-iterated spreading model of $X$ is reflexive.
Then every space generated by a $(k+1)$-iterated spreading model of
$X$ is isomorphic to the space generated by an unconditional
$(k+1)$-spreading model of $X$.
\end{prop}
\begin{proof}
 We first treat the case $k=1$. So  assume that $X$ as well as
every space generated by a  spreading model of $X$ is reflexive. Let
$(\widetilde{e}_n)_n$ be a $2$-iterated spreading model of $X$ and
$\widetilde{E}$  be the  space generated by $(\widetilde{e}_n)_n$.
Also let $\widetilde{E}_0$ be the space generated by a spreading
model of $X$ such that $(\widetilde{e}_n)_n$ is a spreading model of
$\widetilde{E}_0$. Since $X$ is reflexive, by Corollary \ref{cor
singular wrc}, we conclude  that $\widetilde{E}_0$ is isomorphic to
a space $E_0$, generated by an unconditional spreading model of $X$.
Moreover, by our assumption $E_0$ is also reflexive. Summarizing,
the space $E_0$ is reflexive, it  has a Schauder basis which is a
spreading model of $X$ and it  is isomorphic to $\widetilde{E}_0$.
Therefore, $E_0$ admits  a spreading model $(e_n)_n$  equivalent to
$(\widetilde{e}_n)_n$. Let $E$ be the space generated by $(e_n)_n$.
By Lemma \ref{iter_inter}, there exists  an unconditional
$2$-spreading model of $X$ generating a space isomorphic to $E$.
Since $E$ is isomorphic to $\widetilde{E}$ the proof of the
proposition for $ k=1$ is completed.

We proceed by induction. Assume that the proposition holds for some
$k\in\nn$ and let $X$ be a reflexive space such that every space
generated by a $(k+1)$-iterated spreading model of $X$ is reflexive.
Let $(\widetilde{e}_n)_n$ be a $(k+2)$-iterated spreading model of
$X$ and $\widetilde{E}$ be the space that it generates. Let
$\widetilde{E}_0$ be the space generated by a $(k+1)$-iterated
spreading model of $X$ admitting $(\widetilde{e}_n)_n$ as a
spreading model. Since the $k$-iterated spreading models of $X$ are
included in the $(k+1)$-iterated ones, we have that the  spaces
generated by the $k$-iterated spreading models of $X$ are reflexive.
Hence, by our  assumption that the proposition holds for the
positive integer $k$, we have that $\widetilde{E}_0$ is isomorphic
to some space $E_0$ generated by an unconditional $(k+1)$-spreading
model of $X$. Therefore, $E_0$ is  reflexive, it is generated by a
Schauder basic  $(k+1)$-spreading model of $X$ and admits a
spreading model $(e_n)_n$ equivalent to $(\widetilde{e}_n)_n$. Let
$E$ be the space generated by $(e_n)_n$. By Lemma \ref{iter_inter},
there exists  an unconditional $k+2$-spreading model of $X$
generating a space isomorphic to $E$. Since $E$ is isomorphic to
$\widetilde{E}$ the proof of  is completed.
\end{proof}

\begin{cor}\label{qwqwe}
  Let $X$ be a reflexive space such that for every $k\in\nn$, every
  space generated by an unconditional $k$-spreading model of $X$ is
  reflexive. Then for every $k\in\nn$, every space generated by a
  $k$-iterated spreading model of $X$ is isomorphic to the space
  generated by an unconditional $k$-spreading model of $X$.
\end{cor}
\begin{proof}
By Corollary \ref{cor singular wrc} we have that every space
generated by a spreading model of $X$ is isomorphic to the space
generated by an unconditional spreading model of $X$ and therefore
it is reflexive. The proof is carried out by induction and using
Proposition \ref{iterated_lemma}.
\end{proof}


\begin{rem}
  As it is well known, see \cite{BL},
  every non trivial spreading model of $c_0$ generates a space
  isomorphic to $c_0$. This easily implies that every $k$-iterated
  spreading model of $c_0$ generates a space isomorphic to $c_0$.
  On the other hand, as we will see in Section \ref{spreading models of c_0 and
  l^p}, the class of the $2$-spreading models of $c_0$ includes
  all spreading bimonote Schauder basic sequences yielding
  the existence of 2-spreading models which are not 2-iterated ones.
\end{rem}

\begin{rem}
  H.P. Rosenthal had asked whether every 2-iterated spreading
  model of a Banach space $X$ is actually a classical one. In
  \cite{BM} a Banach space $X$ has been constructed not admitting
  $\ell^1$ as a spreading model, while there is a spreading model
  generating a space which contains $\ell^1$. Thus $\ell^1$ occurs
  as 2-iterated spreading model but not as a classical one. A more
  striking result (see \cite{AOTS}) asserts the existence of a
  Banach space $X$ not admitting $\ell^1$ as a spreading model but
  $\ell^1$ is isomorphic to a subspace of every space generated
  generated by a non trivial spreading model of $X$. It remains
  open if for every $k\in\nn$ there exists a Banach space $X_{k+1}$
  such that the class of $(k+1)$-iterated spreading models strictly
  includes the corresponding one of $k$-iterated.
\end{rem}


\section{$k$-spreading models equivalent  to the  $\ell^1$ basis}\label{chapter 9}


In this section we study the properties of the $k$-spreading models equivalent to the usual basis of $\ell^1$.

\subsection{Splitting spreading sequences equivalent to the  $\ell^1$ basis} In this subsection
we present some stability properties of spreading sequences in
seminormed linear spaces which are actually related to the non
distortion of $\ell^1$ (c.f. \cite{J2}).

Let $(e_n)_{n}$ be a spreading sequence in a seminormed linear
space $(E, \|\cdot\|_*)$ and $c>0$. We say that $(e_n)_n$
\emph{admits a lower $\ell^1$-estimate of constant} $c$, if  for
every $n\in\nn$ and $a_1,\ldots,a_n\in\rr$, we have $c\sum_{i=1}^n
|a_i|\leq \big\|\sum_{i=1}^n a_i e_i\big \|_*$.

\begin{prop}\label{propsplitl1}
Let $(E,\|\cdot\|_\circ),(E_1,\|\cdot\|_*),(E_2,\|\cdot\|_{**})$ be
seminormed linear spaces and   $(e_n)_{n},(e_n^1)_{n}$ and
$(e_n^2)_{n}$ be spreading sequences in $E, E_1$ and $E_2$
respectively. Assume that for every $n\in\nn$ and
$a_1,\ldots,a_n\in\rr$, we have
\begin{equation}\label{gt1}\Big\|\sum_{i=1}^na_ie_i\Big\|_\circ\leq
\Big\|\sum_{i=1}^na_ie_i^1\Big\|_*+\Big\|\sum_{i=1}^na_ie_i^2\Big\|_{**}\end{equation}
If $(e_n)_n$ admits a lower $\ell^1$-estimate of constant $c>0$ and
$(e_n^2)_{n}$ does not admit any  lower $\ell^1$-estimate then
$(e_n^1)_{n}$ admits a lower $\ell^1$-estimate of the same constant
$c$.
\end{prop}
\begin{proof}
Suppose on the contrary that $(e_n^1)_{n}$ does not admit a lower
$\ell^1$-estimate of constant $c$. Then there exist $\varepsilon
>0$, $n\in\nn$ and $a_1,\ldots,a_n\in\rr$ with $\sum_{i=1}^n|a_i|=1$
such that $\|\sum_{i=1}^n a_i e_i^1\|_*<c-\varepsilon$. Also  since
$(e_n^2)_{n}$ does not admit any lower $\ell^1$-estimate, there
exist $m\in\nn$ and $b_1,\ldots,b_m\in\rr$ such that
$\sum_{j=1}^m|b_j|=1$ and $\| \sum_{j=1}^mb_j e_j^2
\|_{**}<\varepsilon/2$. Hence,  we get that
\begin{equation}\label{gt2}\begin{split} \Big{\|}\sum_{i=1}^n\sum_{j=1}^m a_i\cdot
b_j e^1_{(i-1)m +j}\Big{\|}_*& \leq \sum_{j=1}^m |b_j|\Big{\|}
\sum_{i=1}^na_i e^1_{(i-1)m +j}\Big{\|}_* < c-\varepsilon
\end{split}\end{equation}
and similarly \begin{equation}\label{gt3}\begin{split}
\Big{\|}\sum_{i=1}^n\sum_{j=1}^m a_i\cdot b_j e^2_{(i-1)m
+j}\Big{\|}_{**}& \leq \sum_{i=1}^n|a_i|\Big{\|}\sum_{j=1}^m b_j
e^2_{(i-1)m +j}\Big{\|}_{**} < \frac{\varepsilon}{2}
\end{split}\end{equation}
But then by  (\ref{gt1}), we obtain  that
\[\begin{split}\Big{\|}\sum_{i=1}^n\sum_{j=1}^m a_i\cdot b_j e_{(i-1)m +j}\Big{\|}_\circ
&\leq  \Big{\|}\sum_{i=1}^n\sum_{j=1}^m a_i\cdot b_j e^1_{(i-1)m +j}\Big{\|}_*\\
&+ \Big{\|}\sum_{i=1}^n\sum_{j=1}^m a_i\cdot b_j e^2_{(i-1)m
+j}\Big{\|}_{**} \stackrel{(\ref{gt2}),
(\ref{gt3})}{<}c-\frac{\varepsilon}{2}\end{split}\] which since
$\sum_{i=1}^n\sum_{j=1}^m |a_i|\cdot|b_j|=1$, contradicts that
$(e_n)_n$ admits a lower $\ell^1$-estimate of constant $c$.
\end{proof}

\begin{cor}\label{ultrafilter property for ell^1 spreading models}
Let $k\in\nn$ and $(x_s)_{s\in [\nn]^k},(x_s^1)_{s\in
[\nn]^k},(x_s^2)_{s\in [\nn]^k}$ be three $k$-sequences in a Banach
space $X$ such that  for all $s\in[\nn]^k$, $x_s=x_s^1+x_s^2$.
Assume that the $k$-sequences $(x_s)_{s\in[\nn]^k}$,
$(x^1_s)_{s\in[\nn]^k}$ and $(x^2_s)_{s\in[\nn]^k}$ generate the
sequences $(e_n)_{n}$, $(e_n^1)_{n}$ and $(e_n^2)_{n}$ respectively,
as  $k$-spreading models. If $(e_n)_n$ admits a lower
$\ell^1$-estimate of constant $c>0$ and  $(e_n^2)_{n}$ does not
admit any lower $\ell^1$-estimate then $(e_n^1)_{n}$ admits a lower
$\ell^1$-estimate of constant $c$.
\end{cor}
\begin{proof} For every
$n\in\nn$,  $a_1,\ldots,a_n\in\rr$ and $(s_j)_{j=1}^n$ in $[\nn]^k$,
we have \begin{equation}\label{sd}\Big\|\sum_{j=1}^n
a_jx_{s_j}\Big\|\leq \Big\|\sum_{j=1}^n a_jx^1_{s_j}\Big\|+
\Big\|\sum_{j=1}^n a_jx^2_{s_j}\Big\|\end{equation} Let
$(E,\|\cdot\|_\circ),(E_1,\|\cdot\|_*),(E_2,\|\cdot\|_{**})$ be the
seminormed linear spaces with Hamel bases $(e_n)_{n},(e_n^1)_{n}$
and $(e_n^2)_{n}$ respectively. Notice that (\ref{sd}) implies that
(\ref{gt1}) holds and therefore the conclusion   follows by
Proposition \ref{propsplitl1}.
\end{proof}

\subsection{$k$-spreading models almost isometric to the $\ell^1$ basis}
Let $c>0$, $k\in\nn$ and $(x_s)_{s\in[\nn]^k}$ be a $k$-sequence in
a Banach space $X$. We will say that the $k$-sequence
$(x_s)_{s\in[\nn]^k}$ \textit{generates} $\ell^1$ \textit{as a
$k$-spreading model of constant} $c$, if $(x_s)_{s\in[\nn]^k}$
generates a $k$-spreading model $(e_n)_n$ which admits a lower
$\ell_1$-estimate of constant $c$.

\begin{prop}\label{Prop on almost isometric l^1 spr mod}
Let $X$ be a Banach space and $k\in\nn$. Assume that $X$ admits a $k$-spreading model equivalent to the usual basis of $\ell^1$. Then for every
$\varepsilon>0$ there exists a $k$-sequence $(y_s)_{s\in[\nn]^k}$ in
$X$ with $1-\varepsilon\leq \|y_s\|\leq 1$, for every $s\in
[\nn]^k$, which generates $\ell^1$ as a $k$-spreading model of
constant $1-\varepsilon$.
\end{prop}
\begin{proof}
Let $(e_n)_n$ be a $k$-spreading model of $X$ which is equivalent to the usual basis of $\ell^1$. Also let $c=\inf
\|\sum_{j=1}^na_je_j\|$, taken over all $n\in\nn$ and
$a_1,\ldots,a_n\in\rr$ with $\sum_{j=1}^n|a_j|=1$. Let
$\varepsilon>0$ and choose $0<\varepsilon'<c$, $p\in\nn$ and
$b_1,...,b_p$ in $[-1,1]$ with $\sum_{i=1}^p|b_i|=1$   such that
\begin{equation}\label{re}\frac{c-\varepsilon'}{c+2\varepsilon'}\geq 1-\varepsilon\;\;\text{and}\;\;\;
c\leq \Big\|\sum_{i=1}^p b_i e_i\Big\|\leq  c+\varepsilon'\end{equation} Let $(x_s)_{s\in[\nn]^k}$ be a $k$-sequence is $X$ generating $(e_n)_n$ as a $k$-spreading model. By
passing to an infinite subset $M$ of $\nn$,  we may assume that for
every $n\in\nn$, $a_1,\ldots,a_n\in[-1,1]$ and
$(s_i)_{i=1}^n\in\textit{Plm}_n([M]^k)$ with $ s_1(1)\geq
M(n)$, we have
\begin{equation}\label{qk}\Bigg|\Big\|\sum_{i=1}^na_ix_{s_i}\Big\|-\Big\|\sum_{i=1}^na_i
e_i\Big\|\Bigg|\leq\varepsilon'\sum_{i=1}^n|a_i|\end{equation} Hence by (\ref{re}), for every $(s_i)_{i=1}^p
\in\textit{Plm}_p([M]^k)$ with $s_1(1)\geq M(p)$ we have that
\begin{equation}\label{lk}c-\varepsilon'\leq\Big\|\sum_{i=1}^p b_ix_{s_i}\Big\|\leq c+2\varepsilon'\end{equation}
For every $s=(n_1,...,n_k)\in[\nn]^k$, we set
\begin{equation}\label{mk}y_s=\frac{\sum_{i=1}^pb_i x_{t_i^s}}{c+2\ee'}
\;\text{where}\; t_i^s=\big(M(p\cdot n_j+i-1)\big)_{j=1}^k,\;\text{for all}\;1\leq i\leq p\end{equation}
Notice that
$(t_i^{s})_{i=1}^p\in\textit{Plm}_p([\nn]^k$ and
$t_1^s(1)=M(p\cdot s(1))\geq M(p)$. Hence,  by (\ref{re}) and
(\ref{lk}), it is clear that $1-\varepsilon\leq \|y_s\|\leq 1$.
Moreover, the $k$-subsequence $(y_s)_{s\in[\nn]^k}$ generates
$\ell^1$ as a $k$-spreading model of constant $1-\varepsilon$.
Indeed, let $l\in\nn$, $a_1,\ldots,a_l\in [-1,1]$ and
$(s_j)_{j=1}^l\in\textit{Plm}_l([\nn]^k)$ with $s_1(1)\geq l$.
Notice that $(t_i^{s_1})_{i=1}^p \;^\frown\;...\;^\frown
{(t_i^{s_l})_{i=1}^p}\in\textit{Plm}_{p\cdot l}([\nn]^k)$ and
$t_1^{s_1}(1)=M(p\cdot s_1(1))\geq
    M(p\cdot l)$. Hence,
\[\Big\|\sum_{j=1}^l a_j y_{s_j}\Big\|=
\Big\|\sum_{j=1}^la_j \cdot \sum_{i=1}^p \frac{ b_i
x_{t_i^{s_j}}}{c+2\ee'}\Big\| \stackrel{(\ref{qk})}{\geq}
\frac{c-\varepsilon'}{c+2\ee'}\sum_{j=1}^l\sum_{i=1}^p|a_j|\cdot|b_i|
\stackrel{(\ref{re})}{\geq}(1-\varepsilon)\sum_{j=1}^l|a_j|\] and the proof is complete.
\end{proof}
\begin{rem}\label{Rem on almost isometric l^1 spr mod}
If we additionally assume that $X$ has a Schauder basis and
$(x_s)_{s\in[M]^k}$ is plegma block (resp. plegma disjointly
supported) then by (\ref{mk}) it is easy to see  that
$(y_s)_{s\in[L]^k}$ is also plegma block (resp. plegma disjointly
supported).
\end{rem}

\subsection{Plegma block generated $k$-spreading models equivalent to the $\ell^1$ basis}
It well known that if a Banach space $X$ with a Schauder basis admits an $\ell^1$ spreading model, then $X$ contains a block sequence which generates an $\ell^1$ spreading model. In this subsection we extend this result. More precisely, we have the following.
\begin{thm}\label{getting block generated ell^1 spreading model}
Let $X$ be a Banach space with a Schauder basis and $k\in\nn$.
Suppose that $\mathcal{SM}_k^{wrc}(X)$ contains up to equivalence
the usual basis of $\ell^1$.  Then there exists a plegma block
generated $k$-spreading model of $X$ equivalent to the usual basis
of $\ell^1$.
\end{thm}
\begin{proof}
Let $k_{X}$ be the minimum  of all $k\in\nn$ such that the set
$\mathcal{SM}_k^{wrc}(X)$ contains a sequence equivalent to the
usual basis of $\ell^1$.
By Remark \ref{remnj}, it suffices to show that $\mathcal{SM}_{k_{X}}(X)$ contains a sequence equivalent to the usual basis of $\ell_1$ which is plegma block generated. For  $k_{X}=1$ this is a well known standard fact. So suppose that
$k_{X}=k\geq 2$ and let $(e_n)_n\in\mathcal{SM}_{k}^{wrc}(X)$ be
equivalent to the usual basis of $\ell^1$. By Corollary \ref{cor
canonical tree with spr mod}, we may assume that $(e_n)_n$ is
generated as a $k$-spreading model by a $k$-sequence
$(x_s)_{s\in[\nn]^{k}}$ which   is subordinated with respect to the
weak topology of $X$ and  admits a canonical tree decomposition
$(y_t)_{t\in[\nn]^{\leq k}}$.

Let $(w_v)_{v\in [\nn]^{k-1}}$ be the $(k-1)$-sequence in  $X$
defined by  $w_v=\sum_{t\sqsubseteq v}y_t$, for every
$v\in[\nn]^{k-1}$. Also let $(x'_s)_{s\in [\nn]^k}$, be the
$k$-sequence defined by $x'_s=w_{s|k-1}$, for every $s\in [\nn]^k$.
Notice that $(w_v)_{v\in [\nn]^{k-1}}$ is subordinated with respect
to the weak topology. Hence  $(w_v)_{v\in [\nn]^{k-1}}$ is a weakly
relatively compact $(k-1)$-sequence. Also, by Lemma \ref{propxi} we
have that $(w_v)_{v\in [\nn]^{k-1}}$ and  $(x'_s)_{s\in [\nn]^k}$
admit the same $(k-1)$-spreading models. Therefore, since the usual
basis of $\ell_1$ is not contained up to equivalence in
$\mathcal{SM}_{k-1}^{wrc}(X)$, we conclude that  $(x'_s)_{s\in
[\nn]^k}$ does not admit a $k$-spreading model equivalent to the
usual basis of $\ell^1$. Since $x_s=x'_s+y_s$, for all $s\in [\nn]^k$, by
Corollary \ref{ultrafilter property for ell^1 spreading models}, we
get  that the $k$-sequence $(y_s)_{s\in[\nn]^{k}}$ admits   a
$k$-spreading model equivalent to the usual basis of $\ell^1$. Since
$(y_s)_{s\in[\nn]^{k}}$ is a plegma block $k$-sequence in $X$ (see
Proposition \ref{trocan} (iii)), the proof is
complete.
\end{proof}

\subsection{Duality of $c_0$ and $\ell^1$ $k$-spreading models}
It is well known that if a Banach space $X$ admits a $c_0$ spreading model, then $X^*$ admits an $\ell^1$ spreading model. In this subsection we extend this result.
\begin{lem}\label{lem for duality}
Let $X$ be a Banach space with a Schauder basis, $k\in\nn$ and
$(x_s)_{s\in[\nn]^{k}}$ be a $k$-sequence  in $X$ which admits a
canonical tree decomposition $(y_t)_{t\in[\nn]^{\leq k}}$ and
generates  a $k$-spreading model equivalent to the usual basis of
$c_0$. Then $y_\emptyset=0$ and there exist $1\leq j_0\leq k$ and
$L\in[\nn]^\infty$ such that the  $k$-subsequence
$(y_{s|j_0})_{s\in[L]^k}$ is plegma block and generates $c_0$ as a
$k$-spreading model.
\end{lem}
\begin{proof}
Since $(x_s)_{s\in[\nn]^{k}}$ generates a $k$-spreading model, we
have that $(x_s)_{s\in[\nn]^{k}}$ is seminormalized. Let $(e_n)$ be
the $k$-spreading model of $(x_s)_{s\in [\nn]^k}$. Since $(e_n)_n$
is equivalent to the usual basis of $c_0$, we have that $(e_n)_n$ is
Cesaro summable to zero. Using these observations we may easily
conclude that $y_\emptyset=0$. We also observe that there exists
$\delta>0$ such that for every  $s\in[\nn]^k$ there exists $1\leq
j\leq k$ such that $\|y_{s|j}\|>\delta$. Hence by Ramsey's theorem
there exists $1\leq j_0\leq k$ and $L\in[\nn]^\infty$ such that for
every $s\in[L]^k$, $\|y_{s|j_0}\|>\delta$.

Let  $n\in\nn$, $a_1,\ldots,a_n\in\rr$ and
$(s_i)_{i=1}^n\in\textit{Plm}_n([L]^k)$. If $I$ is the interval of
$\nn$ with $\min I=\min\text{supp}(y_{s_1|j_0})$ and $\max I= \max
\text{supp}(x_{s_n|j_0})$, then   Proposition \ref{trocan} (v) and
the fact that $y_\emptyset=0$, yield that
\[I\Big(\sum_{i=1}^na_jx_{s_i}\Big)=\sum_{i=1}^na_iy_{s_i|j_0}\]
Hence if $C$ is the basis constant of the Schauder basis of $X$, we
get that
\[\frac{\delta}{2C}\max_{1\leq i\leq n}|a_i|\leq\Big\|\sum_{i=1}^na_iy_{s_i|j_0}\Big\|\leq 2C\Big\|\sum_{i=1}^na_ix_{s_i}\Big\|\]
Therefore, since $(x_s)_{s\in[L]^{k}}$ generates $c_0$ as a
$k$-spreading model, we conclude that every
$k$-spreading model of $(y_{s|j_0})_{s\in [L]^k}$ is equivalent to the usual basis of
$c_0$.
\end{proof}
The above lemma  shows  that the analogue  of Theorem \ref{getting
block generated ell^1 spreading model} for the $c_0$ basis  also
holds. Namely we have the following.
\begin{cor}\label{getting block generated c_0 spreading model}
Let $X$ be a Banach space with a Schauder basis and $k\in\nn$.
Suppose that $\mathcal{SM}_k^{wrc}(X)$ contains up to equivalence
the usual basis of $c_0$.  Then there exists a plegma block
generated $k$-spreading model of $X$ equivalent to the usual basis
of $c_0$.
\end{cor}
\begin{thm}
Let $X$ be a Banach space. Assume that for some $k\in\nn$ the set
$\mathcal{SM}_k^{wrc}(X)$ contains a sequence equivalent to
  the usual basis of $c_0$. Then $X^*$ admits $\ell^1$ as a
  $k$-spreading model.
\end{thm}
\begin{proof}
Let $(x_s)_{s\in[\nn]^k}$ be a subordinated $k$-sequence in $X$
generating $c_0$ as a spreading model. Let $Y$ separable subspace of
$X$ containing the $k$-sequence $(x_s)_{s\in[\nn]^k}$ and $T:Y\to
C[0,1]$ an isometry. Notice that $C[0,1]$ is a Banach space with a bimonotone
Schauder basis and $(T(x_s))_{s\in[\nn]^k}$ is subordinated. Let
$(\ee_n)_n$ a null sequence of positive reals. By Theorem
\ref{canonical tree} there exist  $L\in[\nn]^\infty$ and a
$k$-subsequence $(\widetilde{x}_s)_{s\in[L]^{k}}$ in $C[0,1]$
satisfying the following.
\begin{enumerate}
\item[(P1)] $(\widetilde{x}_s)_{s\in[L]^k}$ admits a canonical tree decomposition $(\widetilde{y}_t)_{t\in[M]^{\leq k}}$.
\item[(P2)] For every $s\in[L]^k$, $\|T(x_s)-\widetilde{x}_s\|<\ee_n$, where $\min
  s=L(n)$.
\end{enumerate}
Notice that property (P2) yields that
$(\widetilde{x}_s)_{s\in[L]^{k}}$ generates $c_0$ as a $k$-spreading
model. By Lemma \ref{lem for duality} there exist $M\in[L]^\infty$
and $1\leq j_0\leq k$ such that the plegma block $k$-subsequence
$(\widetilde{y}_{s|j_0})_{s\in[M]^k}$ generates $c_0$ as a
$k$-spreading model. For every $s\in[M]^k$ we pick
$\widetilde{y}_s^*\in S_{C[0,1]^*}$ with
$\widetilde{y}_s^*(\widetilde{y}_{s|j_0})=\|\widetilde{y}_{s|j_0}\|$
and $\text{supp}\;\widetilde{y}_s^*\subseteq \text{range}\;
\widetilde{y}_{s|j_0}$. For every $s\in[M]^k$ we set
$y^*_s=T^*(\widetilde{y}_s^*)$ and we choose $x^*_s$ in $X^*$ an
extension of $y^*_s$ of the same norm. It is easy to check that
$(x^*_s)_{s\in[M]^k}$ admits $\ell^1$ as a spreading model.
\end{proof}
\section{$k$-Ces\`aro summability vs $\ell^1$ $k$-spreading
models} In this  section we extend the well known dichotomy of
H.P. Rosenthal concerning Ces\`aro summability and $\ell^1$
spreading models (see also  \cite{AT}, \cite{M}). We start by
introducing the definition of the Ces\`aro summability for
$k$-sequences in Banach spaces.

\subsection{Definition of the $k$-Ces\`aro summability in  Banach
spaces}
\begin{defn}\label{Cesaro}
Let $X$ be a Banach space, $x_0\in X$ $k\in\nn$, $(x_s)_{s\in
[\nn]^k}$ be a $k$-sequence in $X$ and $M\in [\nn]^\infty$. We
will say that the $k$-subsequence $(x_s)_{s\in [M]^k}$ is
$k$-Ces\`aro summable to $x_0$ if
\[ \Big(\substack{n\\ \\k}\Big)^{-1} \sum_{s\in [M|n]^k} x_s \;\;\substack{\|\cdot\| \\ \longrightarrow \\ n\to\infty}\;\; x_0\]
where $M|n=\{M(1),...,M(n)\}$.
\end{defn}
\begin{prop}\label{rem on k-Cesaro summability}
Let $X$ be a Banach space, $x_0\in X$, $k\in\nn$,  $(x_s)_{s\in
[\nn]^k}$ be a $k$-sequence in $X$ and $M\in [\nn]^\infty$.
\begin{enumerate}
\item[(i)] If  $(x_s)_{s\in[M]^k}$ norm converges to $x_0$, then
$(x_s)_{s\in [M]^k}$ is $k$-Ces\`aro summable to $x_0$.
\item[(ii)] If $(x_s)_{s\in[M]^k}$ is $k$-Ces\`aro summable to
$x_0$ and in addition it is weakly convergent, then $x_0$ is the
weak limit of $(x_s)_{s\in[M]^k}$. \item[(iii)] If $X^*$ is
separable and for every $N\in[M]^\infty$, $(x_s)_{s\in[N]^k}$ is
$k$-Ces\`aro summable to $x_0$ then there exists $L\in[M]^\infty$
such that $(x_s)_{s\in[L]^k}$ weakly converges to $x_0$.
\end{enumerate}
\end{prop}
\begin{proof}
Assertions  (i) and (ii)  are straightforward. For (iii), first
observe that  for every $x^*\in X^*$, $\ee>0$ and $N\in[M]^\infty$
there exists an $L\in[N]^\infty$ such that
$|x^*(x_s)-x^*(x_0)|<\ee$, for all $s\in[L]^k$. Next for a norm
dense subset $\{x^*_n:n\in\nn\}$ of $X^*$, we inductively choose
an $L\in[M]^\infty$ such that for every $n\in\nn$ and $s\in[L]^k$
with $\min s\geq L(n)$ we have that
$|x_i^*(x_s)-x_i^*(x_0)|<\frac{1}{n}$ for all $1\leq i=1 \leq n$.
This yields that  $(x_s)_{s\in[L]^k}$ weakly converges  to $x_0$.
\end{proof}
\begin{rem}
It is open if assertion (iii) of the above proposition  remains
valid without any restriction for  $X^*$.
\end{rem}
\subsection{A density result for plegma families in $[\nn]^k$}
In this subsection we will present a  density Ramsey result
concerning plegma families. For its proof, we will need the deep
theorem of H. Furstenberg and Y. Katznelson \cite{FK}. Actually,
we shall use the following finite version of this theorem (see
also \cite{G2}).
\begin{thm}\label{Furstenberg and Katznelson Theorem}
Let $k\in\nn$, $F$ be a finite subset of $\mathbb{Z}^k$ and
$\delta>0$. Then there exists   $n_0\in\nn$ such that for all
$n\geq n_0$, every subset $\mathcal{A}$ of $\{1,\ldots,n\}^k$ of
size at least $\delta n^k $ has a subset of the form $a+d F$ for
some $a\in \mathbb{Z}^k$ and $d\in\nn$.
\end{thm}
Our density result for plegma families is the following.
\begin{prop}\label{Lemma using Furstenberg and Katznelson Theorem to find long plegma}
Let  $k, l\in\nn$ and  $\delta>0$. Then there exists $N_0\in \nn$
such that for every $n\geq n_0$ and   every subset $\mathcal{A}$
of $[\{1,\ldots,n\}]^k$ of size at least $\delta (\substack{n\\
k})$, there exists a plegma family
$(s_j)_{j=1}^l\in\textit{Plm}_l([\nn]^k)$ such that
$s_j\in\mathcal{A}$, for every $1\leq j\leq l$.
\end{prop}
 \begin{proof}
For every $1\leq j\leq l$, let $t_j=\big(j,l+j,
2l+j,...,(k-1)l+j\big)$. Clearly
$(t_j)_{j=1}^l\in\textit{Plm}_l([\nn]^k)$. We set
$F=\{\textbf{0}\}\cup\{t_j: \;1\leq j\leq l\}$, where
$\textbf{0}=(0,...,0)$ is the zero element of $\mathbb{Z}^k$. Fix
$\delta>0$. Since $\lim_{n}(\substack{n\\k})/n^k=1/k!$, there
exists $m_0\in\nn$ such that for every $n\geq m_0$ and every
subset $\mathcal{A}$ of $[\{1,\ldots,n\}]^k$ of size at least
$\delta (\substack{n\\ k})$ has density at least
$\frac{\delta}{2k!}$ in $\{1,\ldots,n\}^k$. Hence, by Theorem
\ref{Furstenberg and Katznelson Theorem} (applied for
$\frac{\delta}{2k!}$ in place of $\delta$) we have that there
exists $n_0\geq m_0$ such that for every $n\geq n_0$, every subset
$\mathcal{A}$ of $[\{1,\ldots,n\}]^k$ of size at least $\delta
(\substack{n\\k})$ has a subset of the form $a+d F$ for some $a\in
\mathbb{Z}^k$ and $d\in\nn$. Notice that
$a=a+d\textbf{0}\in\mathcal{A}$ and therefore $a\in
[\{1,...,n\}]^k$. For every $j\in\{1,...,l\}$, we set
$s_j=a+dt_j$. Then $ \{s_j:\;1\leq j\leq l\}\subseteq
\mathcal{A}$. Moreover, since $a\in [\nn]^k$ and $d\in\nn$, we
easily conclude that $(s_j)_{j=1}^l\in\textit{Plm}_l([\nn]^k)$
and the proof is complete.
\end{proof}
\begin{rem}
It is easy to see that for $k=1$ the preceding lemma   trivially
holds (it suffices to set $N_0=\lceil\frac{l}{\delta}\rceil$) and
therefore Theorem \ref{Furstenberg and Katznelson Theorem} is
actually used  for $k\geq 2$. However, it is not completely clear
to us if the full strength of such a deep theorem like
Furstenberg-Katznelson's is actually necessary for the proof of
Proposition \ref{Lemma using Furstenberg and Katznelson Theorem to
find long plegma}.
\end{rem}

\subsection{The main results}
\begin{prop}\label{Furstemberg's application on spreading models}
Let $X$ be a Banach space, $k\in\nn$ and  $(x_s)_{s\in[\nn]^k}$ be
a bounded $k$-sequence in $X$. Let $M\in[\nn]^\infty$ such that
the subsequence $(x_s)_{s\in[M]^k}$ generates a Ces\`aro summable
to zero $k$-spreading model $(e_n)_n$. Then for every
$L\in[M]^\infty$ the $k$-subsequence $(x_s)_{s\in[L]^k}$ is
$k$-Ces\`aro summable to zero.
\end{prop}
\begin{proof}
Assume on the contrary that there exists $L\in[M]^\infty$ such
that $(x_s)_{s\in[L]^k}$ is not $k$-Ces\`aro summable to zero.
Then there exists a $\theta>0$ and a strictly increasing sequence
$(p_n)_n$ of natural numbers such that for every $n\in\nn$,
\begin{equation}\label{wq}\Big(\substack{p_n\\ \\ k}  \Big)^{-1}\Big{\|}\sum_{s\in
[L|p_n]^k} x_s\Big{\|}>\theta\end{equation} For each $n\in\nn$, we pick $x_n^*\in S_{X^*}$ such that $x_n^*\big( (\substack{p_n\\ \\ k})^{-1} \sum_{s\in
[L|p_n]^k} x_s  \big)>\theta$ and we
set \begin{equation}\label{yt}\mathcal{A}_n=\Big\{s\in
\big[\{1,...,p_n\}\big]^k:\;x_n^*(x_{L(s)})>\frac{\theta}{2}\Big\}\end{equation} where
$S_{X^*}$ is the unit sphere of $X^*$. By  (\ref{wq}) and a simple
averaging argument we easily  derive that $|\mathcal{A}_n|\geq
\frac{\theta}{2K} \Big(\substack{p_n\\ \\ k}  \Big)$, where
$K=\sup\{\|x_s\|:s\in [\nn]^k\}$.

We fix $m\in\nn$. By Proposition \ref{Lemma using Furstenberg and
Katznelson Theorem to find long plegma}, with
$\delta=\frac{\theta}{2K}$ and $l=2m-1$,  there exists $n_0\in\nn$
such that for every $n\geq n_0$ there exists a plegma family
$(s_j)_{j=1}^l\in\textit{Plm}_l([\nn]^k)$ such that
$\{s_j:\;1\leq j\leq l\}\subseteq \mathcal{A}_{n}$. Therefore
setting $t_i=L(s_{m+i-1})$ for all $1\leq i\leq m$, we conclude
that for every $m\in\nn$ there exists
$(t_i)_{i=1}^m\in\textit{Plm}_m([L]^k)$ such that
$t_1(1)\geq L(m)$ and $\Big\|\frac{1}{m}\sum_{j=1}^m
x_{t_j}\Big\|>\frac{\theta}{2}$. This easily yields  that $(e_n)_n$ is not Ces\`aro summable to zero, which is a contradiction.
\end{proof}
\begin{cor}\label{Rosenthal prop}
Let $X$ be a Banach space, $k\in\nn$ and  $(x_s)_{s\in[\nn]^k}$ be
a bounded $k$-sequence in $X$. Let $M\in[\nn]^\infty$ such that
the subsequence $(x_s)_{s\in[M]^k}$ generates an unconditional
$k$-spreading model $(e_n)_n$. Then at least one of the following
holds:
\begin{enumerate}
\item[(1)] The sequence  $(e_n)_n$ is equivalent to the usual
basis of $\ell^1$. \item[(2)] For every $L\in [M]^\infty$
 $(x_s)_{s\in [L]^k}$ is $k$-Ces\`aro summable to zero.
\end{enumerate}
\end{cor}
\begin{proof}
Assume that $(e_n)_n$ is not equivalent to the usual basis of
$\ell^1$. Since $(e_n)_n$ is an unconditional spreading sequence,
by Proposition \ref{equiv forms for 1-subsymmetric weakly null} we
have that $(e_n)_n$ is Ces\`aro summable to zero. Hence, by
Proposition \ref{Furstemberg's application on spreading models} we
have that $(x_s)_{s\in [L]^k}$ is $k$-Ces\`aro summable to zero,
for every $L\in [M]^\infty$.
\end{proof}
\begin{rem}  Notice  that in the case $k=1$ the two alternatives
of Corollary \ref{Rosenthal prop} are mutually exclusive. This
does not remain valid for $k\geq 2$. For instance,  assume that
in Example \ref{example}, $(e_n)_n$ is the usual basis of
$\ell^1$. Then the basis $(x_s)_{s\in[\nn]^{k+1}}$ of $X$
generates a $(k+1)$-spreading model equivalent to the usual basis of
$\ell^1$ and simultaneously for every $L\in[\nn]^\infty$,
$(x_s)_{s\in[L]^{k+1}}$ is $(k+1)$-Ces\`aro summable to zero.
Indeed, let $L\in[\nn]^\infty$ and $n\in\nn$. Then since every
plegma tuple in $[L|n]^{k+1}$ is of size less than $n$, we have
\[\Big\|\Big(\substack{n\\\;\\k+1}\Big)^{-1}\sum_{s\in [L|n]^{k+1}} x_s\Big\|_{k+1}\leq n\Big(\substack{n\\\;\\k+1}\Big)^{-1}\]
Since $k+1\geq 2$, $\lim_n n(\substack{n\\k+1})^{-1}=0$. Thus for
every $L\in[\nn]^\infty$, $(x_s)_{s\in [L]^{k+1}}$ is Ces\`aro
summable to zero.
\end{rem}

\begin{thm}\label{Rosenthal thm}
Let $X$ be a Banach space, $k\in\nn$ and  $(x_s)_{s\in[\nn]^k}$ be
a weakly relatively compact $k$-sequence in $X$. Then there exists
$M\in [\nn]^\infty$ such that at least one of the following holds:
\begin{enumerate}
\item[(1)] The subsequence  $(x_s)_{s\in[M]^k}$ generates a
$k$-spreading model equivalent to the usual basis of $\ell^1$.
\item[(2)] There exists $x_0\in X$ such that for every $L\in
[M]^\infty$, $(x_s)_{s\in [L]^k}$ is $k$-Ces\`aro summable to
$x_0$.
\end{enumerate}
\end{thm}
\begin{proof}
First we notice that if there exists $M\in [\nn]^\infty$ such that
$(x_s)_{s\in[M]^k}$ norm converges to some $x_0\in X$, then  by
Proposition  \ref{rem on k-Cesaro summability} (i), we immediately
get that  (2) holds. So we may suppose for the sequel that the
$k$-sequence $(x_s)_{s\in[\nn]^k}$ does not contain any norm
convergent $k$-subsequence.

Let $M_1\in[\nn]^\infty$ such that  $(x_s)_{s\in[M_1]^k}$
generates a $k$-spreading model $(e_n)_n$. By Proposition
\ref{Create subordinated} there exists $M_2\in[M_1]^\infty$ such
that $(x_s)_{s\in[M_2]^k}$ is subordinated (with respect to the
weak topology). Let $\widehat{\varphi}:[M_2]^k\to (X,w)$ be the
continuous map witnessing this and
$x_0=\widehat{\varphi}(\emptyset)$.

For every $s\in [M_2]^k$ we set $x'_s=x_s-x_0$. Notice that the
map $\widehat{\psi}:[M_2]^k\to (X,w)$ defined by
$\widehat{\psi}(t)=\widehat{\varphi}(t)-x_0$ is continuous. Hence
$(x'_s)_{s\in[M_2]^k}$ is subordinated. Since
$\widehat{\psi}(\emptyset)=0$, by Proposition \ref{subordinating
yields convergence}, we have that $(x'_s)_{s\in[M_2]^k}$ is weakly
null. Moreover, since $(x_s)_{s\in[\nn]^k}$ does not contain any
norm convergent $k$-subsequence, it is easy to see that
$(x'_s)_{s\in[M_2]^k}$ is seminormalized.

Let  $(e'_n)_{n}$ be a $k$-spreading model of $(x'_s)_{s\in
[M_2]^k}$ and let $M\in [M_2]^\infty$ such that
$(x'_s)_{s\in[M]^k}$ generates $(e'_n)_{n}$. By Theorem
\ref{unconditional spreading model}, $(e'_n)_{n}$ is unconditional
and therefore, by Corollary \ref{Rosenthal prop}, we have that
either $(e'_n)_{n}$ is equivalent to the usual basis of $\ell^1$
or for every $L\in [M]^\infty$, $(x'_s)_{s\in [L]^k}$ is
$k$-Ces\`aro summable to zero. Since $x_s=x'_s+x_0$, for every
$s\in [M]^k$,  by Lemma \ref{triv-ell} we have that the first alternative yields that
$(e_n)_n$ is equivalent to the usual basis of $\ell^1$ while  the
second one, easily gives that for every $L\in [M]^\infty$,
$(x_s)_{s\in [L]^k}$ is $k$-Ces\`aro summable to $x_0$.
  \end{proof}

\section{The  $k$-spreading models of $c_0$ and $\ell^p$, $1\leq p<\infty$}
\label{spreading models of c_0 and l^p} In this section we deal
with a natural problem, posed to us by Th. Schlumprecht, of
determining the spreading models of the classical sequence spaces.
As we will see, while the spreading models of $\ell^p$, $1\leq
p<\infty$, are as expected, the class of the 2-spreading models of
$c_0$ is surprising large.

\subsection{The $k$-spreading models of $c_0$}
It is well known that every non trivial spreading model of $c_0$
 generates a space isomorphic to $c_0$.  On the other hand the class of the 2-spreading models of
$c_0$ is quite large. As we will see $\mathcal{SM}_2(c_0)$
contains all bimonotone Schauder basic spreading sequences. Notice
that this property of $c_0$ is similar to the one of
$C(\omega^\omega)$ admitting every $1$-unconditional spreading
sequence as a spreading model (see \cite{Odell}).

We start with the following lemma.

\begin{lem}\label{on spr mod of c_0 second lem}
Let $(e_n)_n$ be a spreading sequence in $\ell^\infty$ and let
$(x_s)_{s\in [\nn]^2}$ be the $2$-sequence in $c_0$ defined by
$x_s=(e_{s(1)}(1), e_{s(1)}(2),...,e_{s(1)}(s(2)), 0,0,...)$,  for
every $s\in [\nn]^2$. Then for every non trivial $2$-spreading model
$(\widetilde{e}_n)_{n}$ of $(x_s)_{s\in[\nn]^2}$,  $l\in\nn$ and
$a_1,\ldots,a_l\in\rr$, we have
\begin{equation}\label{gto}\Big\|\sum_{i=1}^la_ie_i\Big\|_\infty\leq\Big\|\sum_{i=1}^la_i\widetilde{e}_i\Big\|\leq\max_{1\leq
j\leq l}\Big\|\sum_{i=j}^l a_i e_i\Big\|_\infty\end{equation}
\end{lem}
\begin{proof}
 We  fix
$l\in\nn$ and $a_1,\ldots, a_l\in\rr$. It is easy to check that for
every $(s_i)_{i=1}^l\in
\textit{Plm}_l([\nn]^k)$, we have that
\begin{equation}\label{kh}\Big\|\sum_{i=1}^la_ix_{s_i}\Big\|_\infty\leq\max_{1\leq j\leq
l}\Big\|\sum_{i=j}^l a_i e_{s_i(1)} \Big\|_\infty\end{equation} Let
$M\in[\nn]^\infty$ such that $(x_s)_{s\in[M]^2}$ generates a non
trivial $2$-spreading model $(\widetilde{e}_n)_{n}$. Then by
(\ref{kh}), we  easily obtain the righthand  inequality of
(\ref{gto}). To complete the proof, we fix $\ee>0$ and  $m_\ee\in\nn$
such that
\begin{equation}\label{pi}\Big\|\sum_{i=1}^la_ie_i\Big\|_\infty-\ee\leq
\Big|\sum_{i=1}^la_ie_i(m_\ee)\Big|\end{equation} Notice that for every
$(s_i)_{i=1}^l\in\textit{Plm}_l([\nn]^2)$ in $[\nn]^2$ with $
s_1(1)\geq m$, we have that
\begin{equation}\label{pii}\Big|\sum_{i=1}^la_ie_i(m)\Big|\leq\Big\|\sum_{i=1}^la_ix_{s_i}\Big\|_\infty\end{equation}
Therefore, since $(x_s)_{s\in[M]^2}$ generates
$(\widetilde{e}_n)_{n}$ as a 2-spreading model, by (\ref{pi}) and
(\ref{pii}), we get that
\[\Big\|\sum_{i=1}^la_ie_i\Big\|_\infty-\ee\leq \Big|\sum_{i=1}^la_ie_i(m_\ee)\Big|\leq \Big\|\sum_{i=1}^la_i\widetilde{e}_i\Big\|_\infty\]
Since this holds for every $\ee>0$, we obtain the lefthand
inequality of (\ref{gto}) and the proof  is complete.
\end{proof}
\begin{prop}\label{c_0 has every bimonotone as order 2}
For every Schauder basic spreading sequence $(e_n)_n$ there exists
$(\widetilde{e}_n)_{n}\in\mathcal{SM}_2(c_0)$ equivalent to
$(e_n)_n$. In particular, if $(e_n)_n$ is bimonotone then
$(e_n)_{n}$ is  contained in $\mathcal{SM}_2(c_0)$.
\end{prop}
\begin{proof}
We may assume that $(e_n)_n$ is a sequence in $\ell^\infty$. Let
$C>0$ be the basis constant of $(e_n)_n$. By Lemma \ref{on spr mod
of c_0 second lem} there exists
$(\widetilde{e}_n)_{n}\in\mathcal{SM}_2(c_0)$ such that for all
$l\in\nn$ and $a_1,\ldots,a_l\in\rr$, we have
\begin{equation}\label{lo}\Big\|\sum_{i=1}^la_ie_i\Big\|_\infty\leq\Big\|\sum_{i=1}^la_i\widetilde{e}_i\Big\|\leq\max_{1\leq
j\leq l}\Big\|\sum_{i=j}^l a_i e_i\Big\|_\infty\leq
(1+C)\Big\|\sum_{i=1}^la_ie_i\Big\|_\infty\end{equation} Hence,
$(e_n)_n$ and $(\widetilde{e}_n)_{n}$ are equivalent. Moreover, if
in addition $(e_n)_n$ is bimonotone then $\text{max}_{1\leq j\leq
l}\|\sum_{i=j}^l a_i e_i\|_\infty\leq \|\sum_{i=1}^l a_i
e_i\|_\infty$ and therefore $(\widetilde{e}_n)_{n}$ is isometric to
$(e_n)_n$.
\end{proof}

\begin{cor}\label{c_0 universcal for singular}
For every singular spreading sequence $(e_n)_n$,  there exists
$(\widetilde{e}_n)_n\in\mathcal{SM}_2(c_0)$ equivalent to
$(e_n)_n$.
\end{cor}
\begin{proof}
Let $e_n=e'_n+e$ be the natural decomposition of $(e_n)_n$. By
Remark \ref{properties of the natural decomposition}, $(e'_n)_n$ is
spreading and 1-unconditional. Hence, by Proposition \ref{c_0 has
every bimonotone as order 2}, there exists a 2-sequence
$(x_s)_{s\in[\nn]^2}$ in $c_0$ generating $(e'_n)_n$ as a
2-spreading model. For every $s\in[\nn]^2$, let $\widetilde{x}_s$ be
the sequence in $c_0$ defined by $\widetilde{x}_s(1)=\|e\|$ and
$\widetilde{x}_s(n+1)=x_s(n)$ for all $n\in\nn$. It is easy to see
that $(\widetilde{x}_s)_{s\in[\nn]^2}$ generates a 2-spreading model
$(\widetilde{e}_n)_n$, satisfying
\[\Big\|\sum_{j=1}^na_j\widetilde{e}_j\Big\|=\max\Big\{\Big|\sum_{j=1}^na_j\Big|\cdot\|e\|,\Big\|\sum_{j=1}^na_je'_j\Big\|\Big\}\]
for all $n\in\nn$ and $a_1,\ldots,a_n\in\rr$. Therefore, by Remark
\ref{properties of the natural decomposition}, we conclude that
$(e_n)_n$ and $(\widetilde{e}_n)_{n}$ are equivalent.
\end{proof}
By Proposition \ref{c_0 has every bimonotone as order 2} and
Corollary \ref{c_0 universcal for singular} we have the following.
\begin{cor}\label{universality_of_c_0}
The set $\mathcal{SM}_2(c_0)$ is isomorphically universal for all
spreading sequences.
\end{cor}

\subsection{The $k$-spreading models of $\ell^p$, for $1\leq
p<\infty$} The $k$-spreading models of the spaces $\ell^p$, for
$1\leq p<\infty$, can be treated as the classical spreading models.
This is based on the observation that the usual basis of these
spaces is symmetric. Therefore, the norm-behavior of the
$k$-sequences admitting a canonical tree decomposition is identical
with the one of sequences being of the form $(x_n+x)_n$,
where $(x_n)_n$ is block.

Especially, for the case of $\ell^1$, one has to make use of the
$w^*$-relative compactness of the bounded $k$-sequences in order to
pass to a subordinated $k$-subsequence with respect to the
$w^*$-topology and in turn to a further one which is approximated by
a $k$-subsequence admitting a canonical tree decomposition. This
procedure yields the following.

\begin{thm}\label{Spr of l^p, 1<p}
Let $1\leq p<\infty$ and $(\widetilde{e}_n)_n$ be a
$k$-spreading model of $\ell^p$, for some $k\in\nn$. Then there
exist $a_1,a_2\geq 0$ such that $(\widetilde{e}_n)_n$ is
isometric to the sequence $(a_1e_1+a_2e_{n+1})_n$, where
$(e_n)_n$ denotes the usual basis of $\ell^p$. More
precisely we have the following.
  \begin{enumerate}
    \item[(i)] The sequence $(\widetilde{e}_n)_n$ is
    trivial if and only if $a_2=0$.
    \item[(ii)] The sequence $(\widetilde{e}_n)_n$ is
    singular if and only if $a_1\neq0$ and $a_2\neq0$.
    \item[(iii)] The sequence $(\widetilde{e}_n)_n$ is
    Schauder basic if and only if $a_1=0$ and $a_2\neq0$. In
    this case $(\widetilde{e}_n)_n$ is equivalent to the
    usual basis of $\ell^p$.
  \end{enumerate}
\end{thm}
\begin{rem}
It can also be shown that  every
$(\widetilde{e}_n)_n\in\mathcal{SM}_{k}^{wrc}(c_0)$, satisfies the
analogue of Theorem \ref{Spr of l^p, 1<p} with $c_0$ in place of
$\ell^p$.\end{rem}
\begin{cor}
 Every non trivial $k$-spreading model of $\ell^p$, $1<p<\infty$,
generates a space isometric to $\ell^p$. In particular, every non
trivial $k$-spreading model of $\ell^1$ is Schauder basic and
equivalent to the usual basis of $\ell^1$.
\end{cor}

\section{A reflexive space not admitting $\ell^p$
 or $c_0$ as a spreading model}\label{space Odel_Schlum}
   A space not admitting any $\ell^p$, for $1\leq p<\infty$, or $c_0$ spreading model was constructed in \cite{O-S}. In the same paper it is asked if there exists a space which does not contain any $\ell^p$, for $1\leq p<\infty$,
    or $c_0$ $k$-iterated spreading model of any $k\in\nn$. In this section we give an example of a reflexive space $X$ answering affirmatively this problem.
\subsection{The definition of the space $X$}
The construction of $X$ is
closely related to the corresponding one in \cite{O-S}.
 Let
$(n_j)_j$ and $(m_j)_j$ be two strictly increasing
sequences of natural numbers satisfying the following:
\begin{enumerate}
  \item[(i)] $\sum_{j=1}^\infty \frac{1}{m_j}\leq 0,1$.
  \item[(ii)] For every $a>0$, we have that
  $\frac{n_j^a}{m_j}\stackrel{j\to\infty}{\longrightarrow}\infty$.
  \item[(iii)] For every $j\in\nn$, we have that $\frac{n_j}{n_{j+1}}<\frac{1}{m_{j}}$.
\end{enumerate}
Let $\|\cdot\|$ be the norm on $c_{00}(\nn)$, implicitly defined
as follows. For every $x\in c_{00}(\nn)$ we set
\begin{equation}\label{eq15}
  \|x\|=\max \Big\{ \|x\|_\infty,\big( \sum_{j=1}^\infty \|x\|_j^2
\big)^\frac{1}{2}\Big\}
\end{equation} where
$\|x\|_j=sup\{\frac{1}{m_j}\sum_{q=1}^{n_j}\|E_q(x)\|:E_1<\ldots<E_{n_j}\}$.

Let $X$ be the completion of $c_{00}(\nn)$ under the above norm.
It is easy to see that the Hamel basis of $c_{00}(\nn)$
is an unconditional basis of the space  $X$. Also notice that for
every $x\in X$ the sequence $w=(\|x\|_j)_j$ belongs to
$\ell^2$ and $\big( \sum_{j=1}^\infty \|x\|_j^2
\big)^\frac{1}{2}=\|w\|_{\ell^2}\leq \|x\|$.
\subsection{The main results}
The following is the main result of this section.
\begin{thm}\label{Odel slumpr theorem}
  For every $k\in\nn$ and $(e_n)_n\in\mathcal{SM}_k(X)$, the space $E$ generated by $(e_n)_n$ does not
  contain any isomorphic copy of $\ell^p$, $1\leq
  p<\infty$, or $c_0$.
\end{thm}
Given the above theorem we get the following consequence which the aforementioned problem stated in \cite{O-S}.
\begin{cor} For every $k\in\nn$, the spaces generated by the $k$-iterated spreading
models of $X$ do not contain any isomorphic copy of $\ell^p$, $1\leq
p<\infty$, or $c_0$.
\end{cor}
\begin{proof}
By Theorem \ref{Odel slumpr theorem} and James' Theorem we have that
for every $k\in\nn$, the spaces generated by the unconditional
$k$-spreading models of $X$ are reflexive. By Corollary \ref{qwqwe}
we have that for every $k\in\nn$, every space generated by a
$k$-iterated spreading model of $X$ is isomorphic to the space
generated by an unconditional $k$-spreading model of $X$. By Theorem
\ref{Odel slumpr theorem}, the proof is complete.
\end{proof}

Also notice this example shows that Krivine's theorem \cite{Kr} concerning
$\ell^p$ or $c_0$ block finite representability cannot be captured by the notion of $k$-spreading models.

\subsection{Proof of Theorem \ref{Odel slumpr theorem}}
We will need the next well known
lemma (see \cite{AT}).
\begin{lem}\label{small_estimation_on_means}
  Let $j<j_0$ in $\nn$ and $(x_q)_{q=1}^{n_{j_0}}$ be a block
  sequence in the unit ball $B_X$ of $X$. Then
  \[\Big\| \frac{x_1+\ldots+x_{n_{j_0}}}{n_{j_0}}
  \Big\|_j<\frac{2}{m_j}\]
\end{lem}

\begin{lem}\label{small means}
  Let $d_0<j_0$ in $\nn$, and $(x_q)_{q=1}^{n_{j_0}}$ be a block sequence in $B_X$. We set $E=\{n\in\nn:n>d_0\}$ and $w_q=(\|x_q\|_j)_j$, for all $1\leq q\leq n_{j_0}$. Assume that for some $0<\ee<1$ there exists a disjointly supported finite sequence $(w'_q)_{q=1}^{n_{j_0}}$ in $\ell^2$ such that $\|E(w_q-w'_q)\|_{\ell^2}<\ee$, for all $1\leq q\leq n_{j_0}$. Then \[\Big\|\frac{x_1+\ldots+x_{n_{j_0}}}{n_{j_0}}\Big\|<0.2+\ee+2n_{j_0}^{-\frac12}\]
\end{lem}
\begin{proof}
   By Lemma
  \ref{small_estimation_on_means}, we have that
  \[\Bigg\|\Big(\Big\|\frac{\sum_{q=1}^{n_{j_0}}x_q}{n_{j_0}}\Big\|_j \Big)_{j=1}^{d_0}\Bigg\|_{\ell^2}
  \leq\sum_{j=1}^{d_0}\Big\|\frac{\sum_{q=1}^{n_{j_0}}x_q}{n_{j_0}}\Big\|_j\leq\sum_{j=1}^{d_0}\frac{2}{m_j}<0,2\]
   Using the above and the observation that $\|E(w'_q)\|_{\ell^2}\leq2$, for all $1\leq q\leq n_{j_0}$, we get the following.
  \[\begin{split}
    \Bigg\|\Big(\Big\|\frac{1}{n_{j_0}}\sum_{q=1}^{n_{j_0}}x_q \Big\|_j\Big)_j\Bigg\|_{\ell_2}&
    \leq0,2+\Bigg\|\Big(\Big\|\frac{1}{n_{j_0}}\sum_{q=1}^{n_{j_0}}x_q\Big\|_j\Big)_{j>d_0}\Bigg\|_{\ell^2}\\
    &\leq0,2+\Bigg\|\frac{1}{n_{j_0}}\sum_{q=1}^{n_{j_0}}\big(w_q(j)\big)_{j>d_0}\Bigg\|_{\ell^2}
    \leq0,2+\Big\|\sum_{q=1}^{n_{j_0}}\frac{E(w'_q)}{n_{j_0}}\Big\|_{\ell^2}+\ee\\
    &\leq0,2+\Big(\sum_{q=1}^{n_{j_0}}\Big(\frac{2}{n_{j_0}}\Big)^2\Big)^\frac{1}{2}+\ee= 0,2+\ee+2n_{j_0}^{-\frac12}
  \end{split}\]
  Moreover  $\|\frac{1}{n_{j_0}}\sum_{q=1}^{n_{j_0}}x_q\|_\infty\leq\frac{1}{n_{j_0}}<\frac{1}{m_1}<0,1$. Hence by (\ref{eq15}) the proof is completed.
\end{proof}

\begin{lem}\label{non containing l^1 block spreading
model}
  For all $k\in\nn$, every plegma block generated $k$-spreading model of $X$ is not equivalent to the usual basis of $\ell^1$.
\end{lem}
\begin{proof}
Assume on the contrary that there exist $k\in\nn$ and a plegma block
$k$-sequence $(x_s)_{s\in[\nn]^k}$ in $X$ which generates $\ell^1$
as a $k$-spreading model. By Proposition \ref{Prop on almost
isometric l^1 spr mod}, we may also assume that $x_s\in B_X$, for
all $s\in[\nn]^k$ and $(x_s)_{s\in[\nn]^k}$ generates $\ell^1$ as a
$k$-spreading model of constant $1-\ee$,  where $\ee=0,1$.

For every $s\in[\nn]^k$,  let $w_s=(\|x_s\|_j)_j$. Since
$(w_s)_{s\in[\nn]^k}$ is a $k$-sequence in $B_{\ell^2}$, it is
weakly relatively compact. Hence, by Proposition \ref{cor for
subordinating}, there exists $M\in[\nn]^\infty$ such that the
$k$-subsequence $(w_s)_{s\in[M]^k}$ is subordinated with respect the
weak topology on $\ell^2$. Let $\widehat{\varphi}:[M]^{\leq
k}\to(\ell^2,w)$ be the continuous map witnessing this. By Theorem
\ref{canonical tree}, there exist $L\in[M]^\infty$ and a
$k$-subsequence $(\widetilde{w}_s)_{s\in[L]^k}$ in $X$ satisfying
the following.
\begin{enumerate}
\item[(i)] $(\widetilde{w}_s)_{s\in[L]^k}$ admits a canonical tree
 decomposition $(\widetilde{z}_t)_{t\in[L]^{\leq k}}$ with
 $\widetilde{z}_\emptyset=\widehat{\varphi}(\emptyset)$.
\item[(ii)] For every $s\in[L]^k$, $\|w_s-\widetilde{w}_s\|_{\ell^2}<\ee/2$, where $\min s=L(n)$.
\item[(iii)] The $k$-subsequence $(\widetilde{w}_s)_{s\in[L]^k}$ is
subordinated with respect to the weak topology of $\ell^2$.
\end{enumerate}
Let $d_0\in\nn$ such that
$\|E(\widehat{\varphi}(\emptyset))\|_{\ell^2}<\frac{\ee}{2}$, where
$E=\{d_0+1,\ldots\}$. For every $s\in[L]^k$ we set
$w'_s=\widetilde{w}_s-\widehat{\varphi}(\emptyset)$. By Proposition
\ref{trocan} (iv), we have that $(w'_s)_{s\in[L]^k}$ is plegma disjointly
supported. Moreover, notice that $\|E(w_s-w'_s)\|_{\ell^2}<\ee$, for
all $s\in[L]^k$. We pick $j_0>d_0$ such that
$2n_{j_0}^{-\frac12}<\ee$. Since $(x_s)_{s\in[\nn]^k}$ generates
$\ell^1$ as a $k$-spreading model of constant $0,9$, we may choose
$(s_q)_{q=1}^{n_{j_0}}\in\textit{Plm}_{n_{j_0}}([L]^k)$ such
that
\begin{equation}\label{eq16}
\Big\| \frac{1}{n_{j_0}}\sum_{q=1}^{n_{j_0}}x_{s_q} \Big\|\geq 0,8
\end{equation}
Observe that $d_0,j_0,\ee$, $(x_{s_q})_{q=1}^{n_{j_0}}$ and
$(w'_{s_q})_{q=1}^{n_{j_0}}$ satisfy the assumptions of Lemma
\ref{small means}. Hence \[\Big\|
\frac{1}{n_{j_0}}\sum_{q=1}^{n_{j_0}}x_{s_q}
\Big\|<0,2+\ee+2n_{j_0}^{-\frac12}<0,4\]
  which contradicts (\ref{eq16}) and the proof is complete.
\end{proof}
\begin{cor} The space  $X$ is reflexive.\end{cor}
\begin{proof}
 Lemma \ref{non containing l^1 block spreading
model} implies that the space $X$ does not contain any isomorphic
copy of $\ell^1$. Moreover, using that
$\frac{n_j}{m_j}\stackrel{j\to\infty}{\longrightarrow}\infty$, it
is easy to see that the space $X$ does not contain any isomorphic
copy of $c_0$. Since the basis of $X$ is unconditional, the result follows by James' theorem.\end{proof}
\begin{cor}\label{corell1} For all $k\in\nn$,
 every  $k$-spreading model of $X$ is not equivalent to the usual basis of $\ell^1$.
 \end{cor}
\begin{proof}
Suppose on the contrary that there exist $k\in\nn$ and a bounded $k$-sequence
$(x_s)_{s\in[\nn]^k}$ which
generates  a $k$-spreading model equivalent to the  $\ell^1$ basis. By the reflexivity of $X$,
 we have that $(x_s)_{s\in[\nn]^k}$ is weakly relatively compact.
Therefore, by Theorem \ref{getting block generated ell^1 spreading
model},  there exists   a plegma block generated
$k$-spreading model of  $X$ equivalent to the usual basis of $\ell^1$, which contradicts to  Lemma \ref{non containing l^1 block spreading
model}.
\end{proof}
\begin{lem}\label{breaking upper l^p}
  Let $1<p\leq\infty$. Then for every $\delta,C>0$ there exists
  $l_0\in\nn$ such that for every $l\geq l_0$ and every block sequence $(x_q)_{q=1}^{n_l}$ in $X$ with $\|x_q\|>\delta$, for all $1\leq q\leq
  n_l$, we have that
  \[\Big\|\sum_{q=1}^{n_l}x_q\Big\|>Cn_l^\frac{1}{p}\]
  where by convection $\frac1\infty=0$
\end{lem}
\begin{proof}
  Since
  $\frac{n_l^{1-\frac{1}{p}}}{m_l}\stackrel{l\to\infty}{\longrightarrow}\infty$,
  there exists $l_0\in\nn$ such
  that $\frac{n_l^{1-\frac{1}{p}}}{m_l}>\frac{C}{\delta}$,  for every $l\geq l_0$. Let $(x_q)_{q=1}^{n_l}$
  be a block sequence in $X$ with $\|x_q\|>\delta$, for all $1\leq q\leq
  n_l$. Then
  \[\Big\|\sum_{q=1}^{n_l}x_q\Big\|\geq\Big\|\sum_{q=1}^{n_l}x_q\Big\|_l\geq
  \frac{1}{m_l}\sum_{q=1}^{n_l}\|x_q\|>\frac{n_l}{m_l}\delta>Cn_l^\frac{1}{p}\]
\end{proof}

\begin{cor}\label{adfs}
 For all $k\in\nn$,
 every  $k$-spreading model of $X$ is not equivalent to the usual basis of $\ell^p$, $1<p<\infty$, or $c_0$.
\end{cor}
\begin{proof}
  Suppose on the contrary that for some $k\in\nn$, $X$ admits a $k$-spreading model
  $(e_n)_n$, which is equivalent to
  the usual basis of either $\ell^p$, for some $1<p<\infty$, or $c_0$. First we shall treat the case of $\ell^p$.
  Since $X$ is reflexive, we have that
  $(e_n)_n\in\mathcal{SM}_k^{wrc}(X)$. By Corollary
  \ref{cor canonical tree with spr mod}, there exists
  a subordinated $k$-sequence $(x_s)_{s\in[\nn]^k}$ admitting a
  canonical tree decomposition $(y_t)_{t\in[\nn]^{\leq k}}$, which
  generates $(e_n)_n$ as a $k$-spreading model. Since the
  basis of $X$ is unconditional and $(e_n)_n$ is Ces\'aro
  summable to zero, it is easy to see that $y_\emptyset=0$.
  Notice that $(x_s)_{s\in[\nn]^k}$ is seminormalized and let $\delta>0$ such that $\|x_s\|>\delta$, for all $s\in[\nn]^k$. Hence,
  for every  $s\in[\nn]^k$ there exists $1\leq d\leq k$ such that
  $\|y_{s|d}\|>\frac{\delta}{k}$. By Ramsey's theorem there exists $1\leq d\leq
  k$ and $L\in[\nn]^\infty$ such that for every $s\in[L]^k$,
  $\|y_{s|d}\|>\frac{\delta}{k}$.
   By Proposition \ref{trocan} (iii), we have that  $(y_{s|d})_{s\in[L]^k}$ is plegma block. Fix $C>0$. By Lemma \ref{breaking upper l^p}
  we have that there exists $l_0$ such that for every $l>l_0$ and $(s_q)_{q=1}^{n_l}\in\textit{Plm}_{n_l}[L]^k$ we have that
  $\Big\|\sum_{q=1}^{n_l}y_{s_q|d}\Big\|>Cn_l^{\frac1p}$.
 Hence, dy the 1-unconditionality of the basis of $X$, we conclude that
    \[\Big\|\sum_{q=1}^{n_l}x_{s_q}\Big\|>Cn_l^{\frac1p}\]
    Since the above holds for every $C>0$ we have that $(e_n)_n$ is not equivalent to the usual basis of $\ell^p$, which is a contradiction.

    Finally, if $(e_n)_n$ is equivalent to the usual basis of $c_0$, then the proof is carried out using identical arguments as above and applying Lemma \ref{breaking upper l^p} for $p=\infty$.
\end{proof}

\begin{proof}[Proof of Theorem \ref{Odel slumpr theorem}]
  Suppose that for some $k\in\nn$ there exists $(e_n)_n\in\mathcal{SM}_{k}(X)$ such that the space $E$ generated by $(e_n)_n$ contains an isomorphic copy of $Y$, where $Y$ is either $\ell^p$, for some $1\leq p<\infty$, or $c_0$. Obviously $(e_n)_n$ is non trivial. Since $X$ is reflexive, $(e_n)_n\in\mathcal{SM}_{k}^{wrc}(X)$. By Corollary \ref{l^p in wrc}, we have that $\mathcal{SM}_{k+1}(X)$ contains a sequence equivalent to the usual basis of $Y$. By Corollaries \ref{corell1} and \ref{adfs}, we get the contradiction.
\end{proof}

\section{A  space $X$ such that $\mathcal{SM}_k(X)$ is a proper subset of $\mathcal{SM}_{k+1}(X)$}\label{s12}
In this section we shall present a Banach space
$\mathfrak{X}_{k+1}$,  having an unconditional basis
$(e_s)_{s\in[\nn]^{k+1}}$ which  generates a $(k+1)$-spreading model
equivalent to the usual basis of $\ell^1$, while the space
$\mathfrak{X}_{k+1}$ does not admit $\ell^1$ as a $k$-spreading
model. Moreover, $(e_s)_{s\in[\nn]^{k+1}}$ is not $(k+1)$-Ces\`aro
summable to any $x_0$ in $\mathfrak{X}_{k+1}$.


\subsection{The definition of the space $\mathfrak{X}_{k+1}$}

We fix for the following a positive integer $k$. We will need the
next definition.
\begin{defn} A family $\mathcal{P}\subseteq [\nn]^{k+1}$ will be
called plegmatic in $[\nn]^{k+1}$, if there exist a finite block
sequence $F_1<\ldots<F_{k+1}$ of subsets of $\nn$ with
$|F_1|=\ldots=|F_{k+1}|$ such that $\mathcal{P}\subseteq
F_1\times\ldots\times F_{k+1}$. A plegmatic family
$\mathcal{P}\subseteq[\nn]^{k+1}$ will be called Schreier if in
addition $|F_1|\leq \min F_1$.
\end{defn}
For instance, for every
$(s_j)_{j=1}^l\in\textit{Plm}_l(\nn]^{k+1}$, the family
$\mathcal{P}=\{s_1,\ldots,s_l\}$ is plegmatic but notice that not
all  plegmatic families in $[\nn]^{k+1}$  are plegma.

Let $(e_s)_{s\in [\nn]^{k+1}}$ be the Hamel basis of
$c_{00}([\nn]^{k+1})$. For every $x=\sum_{s\in[\nn]^{k+1}}x(s)e_s$
in $c_{00}([\nn]^{k+1})$, we set
\begin{equation}\label{no}\|x\|=\sup
\Big(\sum_{i=1}^n\|\mathcal{P}_i(x)
\|_1^2\Big)^{\frac{1}{2}}\end{equation} where
$\|\mathcal{P}(x)\|_1=\sum_{s\in \mathcal{P}}|x(s)|$, for all
$\mathcal{P}\subseteq [\nn]^{k+1}$ and the supremum in (\ref{no}) is
taken over all finite sequences $(\mathcal{P}_i)_{i=1}^n$ of
disjoint Schreier plegmatic families in $[\nn]^{k+1}$. The space
$\mathfrak{X}_{k+1}$ is defined to be the completion of
$(c_{00}([\nn]^{k+1}),\|\cdot\|)$.

The proof of the next proposition is straightforward.

\begin{prop} \label{mn}The Hamel basis $(e_s)_{s\in[\nn]^{k+1}}$
of $c_{00}([\nn]^{k+1})$ is an unconditional basis for the space
$\mathfrak{X}_{k+1}$ and  it  generates a $(k+1)$-spreading model
which is isometric to the usual basis of  $\ell^1$.
\end{prop}

We may also define a norming set $W$ for the space
$\mathfrak{X}_{k+1}$ as follows. First, let
\[W^0=\Big\{
\sum_{s\in\mathcal{P}} \pm e_s^*: \mathcal{P}\subseteq
[\nn]^{k+1}\;\;\text{is Schreier plegmatic}\Big\}\] For each
$f=\sum_{s\in\mathcal{P}} e_s^*\in W^0$,  the support of $f$,
denoted by $\text{supp}(f)$,  is defined to be the family
$\mathcal{P}$. It is easy to see that a norming set for
$\mathfrak{X}_{k+1}$ is the set $W$ which consists  of all
$f=\sum_{i=1}^n\lambda_i f_i$ where $(f_i)_{i=1}^n$ is a sequence
in $W^0$ such that
$\text{supp}(f_i)\cap\text{supp}(f_j)=\emptyset$, for all $1\leq
i<j\leq n$ and $\sum_{i=1}^n\lambda_i^2\leq 1$.

In order to study the basic properties of the space
$\mathfrak{X}_{k+1}$, we  need the following  proposition.
\begin{prop} \label{The space X_k+1 does not contain l^1 disjointly supported spreading models of order k}
Every  plegma disjointly generated $k$-spreading model of
$\mathfrak{X}_{k+1}$ is not equivalent to the usual basis of
$\ell^1$.
\end{prop}
The proof  is postponed in the next subsection. Assuming
Proposition \ref{The space X_k+1 does not contain l^1 disjointly
supported spreading models of order k} we  are able to  prove the
following.

\begin{thm} The space $\mathfrak{X}_{k+1}$ has the next properties.
\item [(i)] It is reflexive.
\item[(ii)] There is no sequence $(e_n)_n\in\mathcal{SM}_{k}(\mathfrak{X}_{k+1})$ equivalent to the usual basis of $\ell^1$.
\item[(iii)] Every (k+1)-subsequence of  $(e_s)_{s\in[\nn]^{k+1}}$  is not
$(k+1)$-Ces\`aro summable to any $x_0$ in $\mathfrak{X}_{k+1}$.
\end{thm}

\begin{proof} (i) By Proposition \ref{mn}, we have that
$(e_s)_{s\in[\nn]^{k+1}}$ is  unconditional. Also, it is easy to
check that it is boundedly complete. Thus  $c_0$ is not contained in
$\mathfrak{X}_{k+1}$. Moreover, the same holds for $\ell^1$, since
otherwise there would exist a disjointly supported sequence
$(x_n)_{n}\in\mathfrak{X}_{k+1}$ equivalent to the usual basis of
$\ell^1$, which is impossible by  Proposition \ref{The space X_k+1
does not contain l^1 disjointly supported spreading models of order
k}. Hence, by James' theorem \cite{J}, the space
$\mathfrak{X}_{k+1}$ is reflexive.\\
(ii) Assume on the contrary, that there exists $(e_n)_{n}$ in
$\mathcal{SM}_k(\mathfrak{X}_{k+1})$ equivalent to the usual basis
of $\ell^1$. Since $\mathfrak{X}_{k+1}$ is reflexive, we get that
$(e_n)_{n}\in\mathcal{SM}^{wrc}_k(\mathfrak{X}_{k+1})$. Hence, by
Corollary \ref{cor canonical tree with spr mod}, $(e_n)_{n}$ is
generated by a $k$-sequence $(x_s)_{s\in[\nn]^k}$ in $X_{k+1}$
admitting a canonical tree decomposition $(y_t)_{t\in[\nn]^{\leq
k}}$. Setting $x'_s=x_s-y_\emptyset$, for all $s\in[\nn]^k$,
 by Lemma \ref{triv-ell}, we have that $(x'_s)_{s\in[\nn]^k}$ also
admits a $k$-spreading model equivalent to the usual basis of
$\ell^1$. Since  $(x'_s)_{s\in[\nn]^k}$ is a plegma disjointly
supported $k$-sequence, by Proposition \ref{The space X_k+1 does not
contain l^1 disjointly supported
spreading models of order k} we have reached to a contradiction.\\
(iii)  Since $\mathfrak{X}_{k+1}$ is reflexive we have that
$(e_s)_{s\in[\nn]^{k+1}}$ is a weakly null $(k+1)$-sequence. Let
$M\in[\nn]^\infty$ and assume that $(e_s)_{s\in[M]^{k+1}}$ is
$(k+1)$-Ces\`aro summable to some $x_0\in\mathfrak{X}_{k+1}$. By
Proposition \ref{rem on k-Cesaro summability}(ii), we get that $x_0=0$.
For every $n\in\nn$,  let \begin{equation}y_n=\Big(
\substack{(k+2)n\\\ k+1} \Big)^{-1}\sum_{s\in[M|
(k+2)n]^{k+1}}e_s\end{equation} where $l_n=(k+2)n$,
$\mathcal{P}_n=F_1^n\times\ldots\times
  F_{k+1}^n$, where for every  $1\leq i\leq
k+1$, $F_i^n=\{M(in+1),\ldots,M((i+1)n)\}$ and $f_n=\sum_{s\in
\mathcal{P}_n}e^*_s$. It is easy to check that
\begin{equation}\label{qr}f_n(y_n)= n^{k+1}\cdot \Big( \substack{(k+2)n\\k+1}
\Big)^{-1}\stackrel{n\to\infty}{\longrightarrow}\frac{(k+1)!}{(k+2)^{k+1}}\end{equation}
 Since $\|y_n\|\geq f_n(y_n)$, by (\ref{qr}) we conclude that
$(e_s)_{s\in[M]^{k+1}}$ is not $(k+1)$-Ces\`aro summable to $x_0=0$,
a contradiction.
\end{proof}

\subsection{Proof of Proposition \ref{The space X_k+1
does not contain l^1 disjointly supported spreading models of order
k}} \begin{lem}\label{supports} Let    $x\in \mathfrak{X}_{k+1}$ of
finite support and  $f\in W^0$ such that $\text{supp} (f)\cap
\text{supp}(x)\neq \emptyset$. Then $|\text{supp}(f)|\leq
n_0^{k+1}$, where $n_0=\max\{s(1):\;s\in\text{supp}(x)\}$.
\end{lem}
\begin{proof}
There exist $F_1<\ldots< F_{k+1}$ subsets of $\nn$ such that
$|F_1|=...=|F_{k+1}|$, $\text{supp}(f)\subseteq
F_1\times\ldots\times F_{k+1}$ and $|F_1|\leq\min F_1$. Hence
$|\text{supp}(f)|\leq(\min F_1)^{k+1}$. Let
$s\in\text{supp}(f)\cap\text{supp}(x)$. Then $n_0\geq s(1)\geq\min
F_1$. Hence $n_0\geq\min F_1$ and therefore $|\text{supp}(f)|\leq
n_0^{k+1}$.
\end{proof}
\begin{lem}
  Let $N_0\in\nn$. Then for every $0<\ee<1$, every $l\in\nn$ and every disjointly supported finite sequence $(x_j)_{j=1}^l$ in the unit ball of $\mathfrak{X}_{k+1}$ such that for every $1\leq j\leq l$ and $s\in\text{supp}(x_j)$, $s(1)\leq N_0$, we have that
\[\Big\|\frac{1}{l}\sum_{j=1}^{l}x_{j}\Big\|\leq \ee +\frac{N_0^{k+1}}{\ee^2 l}\]
\end{lem}

\begin{proof}
  We fix $0<\ee<1$, $l\in\nn$ and $(x_j)_{j=1}^l$ satisfying the assumptions of the lemma.
 Let $\varphi=\sum_{i=1}^n\lambda_if_i\in W$, where $n\in\nn$,
$\lambda_1,\ldots,\lambda_n\in\rr$ with
$\sum_{i=1}^n\lambda_i^2\leq1$ and $f_1,\ldots,f_n\in W^0$ pairwise
disjointly supported. For every $j=1,\ldots,l$ we set
\[I_j=\Big\{i\in\{1,\ldots,n\}:\;\text{supp}(f_i)\cap\text{supp}(x_j)\neq\emptyset \Big\}\]
By Lemma \ref{supports},  we have that for every $1\leq j\leq l$, if $i\in I_j$ then
$|\text{supp}(f_i)|\leq N_0^{k+1}$. Also let
$F_1=\{j\in\{1,\ldots,l\}:\; \sum_{i\in I_j}\lambda_i^2
<\varepsilon^2\}$ and $F_2=\{1,\ldots,l\}\setminus F_1$.
It is easy to see that $\sum_{i\in I_j}\frac{f_i(x_j)}{( \sum_{i\in I_j}f_i(x_j)^2 )^\frac{1}{2}}f_i$ belongs to $W$,  for all $1\leq j\leq l$. Hence, since $\|x_j\|\leq1$, we have that $\sum_{i\in I_j} f_i(x_j)^2\leq1$,  for all $1\leq j\leq l$.
Therefore we have
\[\begin{split}
\varphi\Big(\sum_{j=1}^{l}x_j\Big)&= \sum_{i=1}^n\lambda_i f_i\Big(\sum_{j=1}^{l}x_j\Big) =
    \sum_{j=1}^{l} \sum_{i=1}^n \lambda_i f_i(x_j)\\
    &= \sum_{j=1}^{l} \sum_{i\in I_j} \lambda_i f_i(x_j)\leq
    \sum_{j=1}^{l}
    \Big{(}\sum_{i\in I_j} \lambda_i^2\Big{)}^\frac{1}{2} \Big{(}\sum_{i\in I_j}f_i(x_j)^2\Big{)}^\frac{1}{2}\\
    &\leq  \sum_{j\in F_1} \Big{(}\sum_{i\in I_j} \lambda_i^2\Big{)}^\frac{1}{2}  + \sum_{j\in F_2}
     \Big{(}\sum_{i\in I_j} \lambda_i^2\Big{)}^\frac{1}{2}\\
    &\leq \varepsilon |F_1|+|F_2|\leq \ee l+|F_2|
  \end{split}\]
If for some $1\leq i\leq n$ we have that $J_i\neq\emptyset$ then, by Lemma \ref{supports} we have that $|\text{supp}(f_i)|\leq N_0^{k+1}$ and since  $(x_j)_{j=1}^{l}$ are
disjointly supported, we conclude that  $|J_i|\leq N_0^{k+1}$. Therefore, for every $1\leq i\leq n$, $|J_i|\leq N_0^{k+1}$. Hence
$$\ee^2 |F_2|\leq \sum_{j\in F_2}\sum_{i\in I_j}\lambda_i^2\leq \sum_{j=1}^{l}\sum_{i\in I_j}\lambda_i^2=\sum_{i=1}^n|J_i|\lambda_i^2\leq N_0^{k+1}\sum_{i=1}^n \lambda_i^2\leq
N_0^{k+1}$$ which yields that $|F_2|\leq N_0^{k+1}/\ee^2$.
Therefore, for every $\varphi\in W$ we have
\[\varphi\Big(\sum_{j=1}^{l}x_j\Big)\leq \varepsilon
l+\frac{N_0^{k+1}}{\ee^2}\] Since $W$ is a norming set for $\mathfrak{X}_{k+1}$, the proof is complete.
\end{proof}
\begin{defn} (i) Let $\mathcal{G}_1,\mathcal{G}_2\subseteq
[\nn]^{k+1}$. We will call the pair $(\mathcal{G}_1,\mathcal{G}_2)$
weakly plegmatic if  for every $s_2\in G_{2}$ there exists $s_1\in
\mathcal{G}_i$ such that the
pair $\{s_1,s_2\}$ is  plegmatic.\\
 (ii) For
every $0\leq j\leq l$, let $\mathcal{G}_j\subseteq [\nn]^{k+1}$. The
finite sequence $(\mathcal{G}_j)_{j=0}^l$  will be called a
\textit{weakly plegmatic path of subsets of} $[\nn]^{k+1}$, if for
every $0\leq i< l$ the pair $(\mathcal{G}_{i},\mathcal{G}_{i+1})$ is
weakly plegmatic.
\end{defn}

\begin{lem} \label{blocking the first coordinates in weakly allowable paths}
Let $(\mathcal{G}_j)_{j=0}^k$ be a weakly plegmatic path of subsets
in $[\nn]^{k+1}$. Then $\max\{s(1):s\in\cup_{j=0}^k
\mathcal{G}_j\}\leq \max\{ s(k+1): s\in \g_0\}$.
\end{lem}
\begin{proof}
Let $0\leq j\leq k$ and $s\in \g_j$. Then it is easy to see that there exists a sequence
$(s_i)_{i=0}^j$ in $[\nn]^{k+1}$ with  $s_i\in \g_i$, for every $0\leq
i\leq j-1$ and $s_j=s$, such that $\{s_i,s_{i+1}\}$ is plegmatic, for all $0\leq i\leq j-1$. Hence
$$s(1)=s_j(1)<s_{j-1}(2)<...<s_0(j+1)\leq s_0(k+1)\leq \max\{
s(k+1): s\in \g_0\}$$
\end{proof}
\begin{lem} \label{Lemma for 2 vectos of norm almost 2}
Let $0<\eta<\frac18$ and  $x_1,x_2\in\mathfrak{X}_{k+1}$ with disjoint finite
supports such that $\|x_1\|,\|x_2\|\leq1$ and $\|x_1+x_2\|> 2-2\eta$.
Let  $\mathcal{G}_1\subseteq\text{supp}(x_1)$ such that
$\|\mathcal{G}_1^c(x_1)\|\leq \eta$. Then  there exists
$\mathcal{G}_2\subseteq\text{supp}(x_2)$ satisfying the following.
\begin{enumerate}
\item[(i)] The pair $(\mathcal{G}_1,\mathcal{G}_2)$ is a weakly plegmatic path and
\item[(ii)] $\|\mathcal{G}_2^c(x_2)\|\leq \eta^\frac{1}{8}$.
\end{enumerate}
\end{lem}
\begin{proof}
  Since $\|x_1+x_2\|>2-2\eta$, there exists $\varphi\in W$ such that $\varphi(x_1+x_2)>2-2\eta$.
  Since $\|x_1\|,\|x_2\|\leq1$, we get that $\varphi(x_1)>1-2\eta$ and
  $\varphi(x_2)>1-2\eta$. The functional $\varphi$ is of the form $\sum_{i=1}^n \lambda_i f_i$,
  where $f_1,\ldots,f_n$ are pairwise disjoint supported elements of $W^0$ and $\sum_{i=1}^n\lambda_i^2\leq1$. We set $I=\{1,\ldots,n\}$
  and we split it to $I_1$ and $I_2$ as follows:
  \[I_1=\{i\in I:\;\text{supp}(f_i)\cap \g_1\neq\emptyset\}\;\;\text{and}\;\;I_2=I\setminus I_1=\{i\in I:\;\text{supp}(f_i)\subseteq \g_1^c\}\]
  We also set $\varphi_1=\sum_{i\in I_1}\lambda_i f_i$ and $\varphi_2=\sum_{i\in I_2}\lambda_i f_i$.
   Hence $\varphi_2(x_1)\leq\|\g_1^c(x_1)\|\leq \eta$ and therefore $\varphi_1(x_1)> 1-3\eta$.
   Applying Cauchy-Schwartz's inequality we get that
  \[1-3\eta < \varphi_1(x_1)=\sum_{i\in I_1}\lambda_i f_i(x_1)\leq
  \Big{(} \sum_{i\in I_1}\lambda_i^2 \Big{)}^\frac{1}{2}\Big{(} \sum_{i\in I_1}f_i(x_1)^2
   \Big{)}^\frac{1}{2}\leq\Big{(} \sum_{i\in I_1}\lambda_i^2 \Big{)}^\frac{1}{2}\]
  Since $\sum_{i\in I}\lambda_i^2\leq1$, we have that
$(\sum_{i\in I_2}\lambda_i^2)^\frac{1}{2}< (1-(1-3\eta)^2)^\frac{1}{2}\leq (6\eta)^\frac{1}{2}$.
  Hence
    \[\varphi_2(x_2)=\sum_{i\in I_2}\lambda_if_i(x_2)\leq \Big{(}
     \sum_{i\in I_2}\lambda_i^2 \Big{)}^\frac{1}{2}\Big{(} \sum_{i\in I_2} f_i(x_2)^2
      \Big{)}^\frac{1}{2}< (6\eta)^\frac{1}{2}\]
  Hence
  $\varphi_1(x_2)> 1-2\eta-(6\eta)^\frac{1}{2}>1-4\eta^\frac{1}{2}$.
We set $\g_2=\text{supp}(x_2)\cap\text{supp}(\varphi_1)$.
  Then by the definition of $I_1$ it is immediate that the pair $(\g_1,\g_2)$ is weakly plegmatic.
Finally, since $\|\g_2(x_2)\|^2+\|\g_2^c(x_2)\|^2\leq
\|x_2\|^2\leq 1$ and $\|\g_2(x_2)\|\geq
\varphi_1(x_2)$,
 we get that
 $\|\g_2^c(x_2)\|\leq (1-(1-4\eta^\frac12)^2)^\frac{1}{2}<\eta^\frac18$ and the proof is complete.
\end{proof}
An iterated use of the above yields the following.
\begin{cor}\label{final lemma for mathfrak-X-_k+1}
Let $m\in\nn$ and $0<\ee<\frac18$  Then for every sequence
$(x_i)_{i=0}^m$ of  disjointly and finitely supported  vectors in
$\mathfrak{X}_{k+1}$  with $\|x_i\|\leq1$, for all $0\leq i\leq m$,
and  $\|x_i+x_{i+1}\|> 2-2\varepsilon^{8^m}$, for all $0\leq i<m$,  there
exists a weakly plegmatic path $(\g_i)_{i=0}^m$ of subsets
of $[\nn]^{k+1}$ such that $\mathcal{G}_i\subseteq \text{supp}\;x_i$
and $\|\g_i^c(x_i)\|<\ee$, for all $0\leq i\leq m$.
\end{cor}
We are now ready to give the proof of Proposition \ref{The space X_k+1
does not contain l^1 disjointly supported spreading models of order
k}.
\begin{proof}[Proof of Proposition \ref{The space X_k+1
does not contain l^1 disjointly supported spreading models of order
k}:]
  Assume on the contrary that the space $\mathfrak{X}_{k+1}$ admits
  a plegma disjointly generated $k$-spreading model equivalent to the usual basis of $\ell^1$. Let $0<\ee<\frac18$.
  By Proposition \ref{Prop on almost isometric l^1 spr mod} and
  Remark \ref{Rem on almost isometric l^1 spr mod} there exists a
  sequence
  $(x_t)_{t\in[\nn]^k}$ in the unit ball of $\mathfrak{X}_{k+1}$ which is plegma disjointly supported and generates $\ell^1$ as a $k$-spreading model
  of constant $c>1-\ee^{8^k}$. Therefore, we may suppose that
  \begin{equation}\label{eq14}
    \Big\|\frac1l\sum_{j=1}^lx_{t_j}\Big\|>1-\ee^{8^k}
  \end{equation}
  for all $l\in\nn$ and $(t_j)_{j=1}^l\in\textit{Plm}_l([\nn]^k)$ with $t_1(1)\geq l$.

  We set $t_0=\{2,4,\ldots,2k\}$, $N_0=\max\{s(k+1):s\in\text{supp}(x_{t_0})\}$ and $L=\{2n:s>k\}$. For every $t\in[L]^k$ we select $\g_t\subseteq[\nn]^k$ such that $\g_t\subseteq\text{supp}(x_t)$, $\|\g_t^c(x_t)\|<\ee$ and $s(1)<N_0$, for all $s\in\g_t$, as follows. Let $t\in[L]^k$. Observe $t\in [\nn]_\shortparallel^{k}$ and $t_0<t$. By Proposition \ref{accessing everything with plegma path of
length |s_0|}  there exists a plegma path
    $(t_j)_{j=0}^{k}$ in $[\nn]^{k}$, with $t_k=t$.
    By Corollary \ref{final lemma for mathfrak-X-_k+1} (for $m=k$)
    there exists a weakly plegmatic path $(\g_j)_{j=0}^{k}$ such that  $\g_j\subseteq \text{supp}\;x_{t_j}$ and $\|\g_j^c(x_{t_j})\|<\ee$,
    for all $j=0,\ldots,k$. We set $\g_t=\g_k$. Lemma \ref{blocking the first coordinates in weakly allowable paths} and Corollary \ref{final lemma for mathfrak-X-_k+1} yield that the choice of $(\g_t)_{t\in[L]^k}$ is as desired.

    For every $t\in[L]^k$, let $x_t^1=\g_t(x_t)$. Then $\|x_t-x_t^1\|<\ee$, for all $t\in[L]^k$. Hence by (\ref{eq14}) we get that for every $l\in\nn$ and every $(t_j)_{j=1}^l\in\textit{Plm}_l([L]^k)$
     with $t_1(1)\geq l$, we have that \begin{equation}\label{eq9}
    \Big{\|} \frac{1}{l}\sum_{i=1}^l x^1_{t_i} \Big{\|}>1-2\varepsilon>\frac68
  \end{equation}
  Moreover notice that $(x'_t)_{t\in[L]^k}$ is a plegma disjointly supported $k$-subsequence in the unit ball of $\mathfrak{X}_{k+1}$. Therefore,
  by Lemma \ref{blocking the first coordinates in weakly allowable
    paths} and (\ref{eq9}) for  $l>8N_0^{k+1}/5\ee^2$, we get a contradiction. The proof of Proposition \ref{The space X_k+1
does not contain l^1 disjointly supported spreading models of order
k} is complete. \end{proof}

\begin{rem}
  As we have mentioned in the introduction of this article, the
  $k$-spreading models of a Banach space $X$ have a transfinite
  extension yielding an hierarchy of $\xi$-spreading models, for
  $\xi<\omega_1$. It can be shown that the space in Section \ref{space Odel_Schlum} does not admit $\ell^p$, for $1\leq p<\infty$, or $c_0$ as $\xi$-spreading model, for every $\xi<\omega_1$. Also an analogue of the last example exists. Namely, for every limit countable ordinal $\xi$ there exists a reflexive space $X_\xi$ admitting $\ell^1$ as $\xi$-spreading model but not less.
\end{rem}

\end{document}